\documentclass[11pt]{amsart}

\usepackage{amsmath,amsthm,amssymb,fancyhdr,graphicx,bbm,cancel,color,mathrsfs,todonotes,xypic, pinlabel,microtype,mathtools}
\usepackage[all]{xy}

\usepackage{bm}

\usepackage{stackengine,scalerel}
\usepackage{tikz}
\usepackage{tikz-cd}
\usepackage[capitalise]{cleveref}

\newtheorem{thm}{Theorem}
\newtheorem{lem}[thm]{Lemma}
\newtheorem{prop}[thm]{Proposition}
\newtheorem{cor}[thm]{Corollary}
\theoremstyle{definition} 

\newtheorem{qu}[thm]{Question} 
 
\newtheorem{rmk}[thm]{Remark}
\theoremstyle{remark} 
\newtheorem*{nrem}{Remark}

\newenvironment{defn}[1][Definition.]{\begin{trivlist}
\item[\hskip \labelsep {\bfseries #1}]}{\end{trivlist}}

\newenvironment{rem}[1][Remark.]{\begin{trivlist}
\item[\hskip \labelsep {\bfseries #1}]}{\end{trivlist}}

\newcommand{\dmo}{\DeclareMathOperator}

\newcommand{\R}{\mathbb{R}}
\newcommand{\Q}{\mathbb{Q}}
\newcommand{\co}{\mathbb{C}}\newcommand{\Z}{\mathbb{Z}}

\newcommand{\al}{\alpha}\newcommand{\be}{\beta}\newcommand{\ga}{\gamma}\newcommand{\de}{\delta}\newcommand{\ep}{\epsilon}\newcommand{\si}{\sigma}
\newcommand{\Om}{\Omega}\newcommand{\De}{\Delta}\newcommand{\la}{\lambda}
\newcommand{\Ga}{\Gamma}
\newcommand{\wtil}{\widetilde}\newcommand{\ta}{\theta}\newcommand{\Lam}{\Lambda}

\newcommand{\cd}{\cdots}\newcommand{\ld}{\ldots}
\newcommand{\sbs}{\subset}\newcommand{\bs}{\backslash}\newcommand{\pa}{\partial}

\newcommand{\xra}{\xrightarrow}
\newcommand{\Lra}{\Leftrightarrow}
\newcommand{\ra}{\rightarrow}
\newcommand{\hra}{\hookrightarrow}

\newcommand{\bb}[1]{\mathbb{#1}}\newcommand{\ca}[1]{\mathcal{#1}}\newcommand{\ov}[1]{\overline{#1}}\newcommand{\mf}{\mathfrak}

\newcommand{\fr}[2]{\frac{#1}{#2}}

\newcommand{\ot}{\otimes}
\newcommand{\lan}{\langle}\newcommand{\ran}{\rangle}
\newcommand{\op}{\oplus}\newcommand{\til}{\tilde}
\newcommand{\ti}{\times}
\newcommand{\bsym}{\boldsymbol}

\dmo{\sgn}{sign}
\dmo{\we}{\wedge}
\dmo{\ind}{ind}\dmo{\Ind}{Ind}
\dmo{\bop}{\bigoplus}\dmo{\pic}{Pic}
\dmo{\coker}{coker}\dmo{\vol}{Vol}\dmo{\gal}{Gal}\dmo{\perm}{Perm}
\dmo{\tor}{Tor}\dmo{\ext}{Ext}\dmo{\Ext}{Ext}
\dmo{\aut}{aut}
\dmo{\Aut}{Aut}
\dmo{\inn}{Inn}\dmo{\var}{Var}
\dmo{\dep}{depth}\newcommand{\rest}[2]{#1\bigr\vert_{#2}}
\dmo{\ad}{ad}\dmo{\curl}{curl}
\dmo{\hy}{\bb H}\dmo{\Sl}{SL}
\dmo{\SO}{SO}
\dmo{\oO}{O}
\dmo{\psl}{PSL}
\dmo{\isom}{Isom}\dmo{\Isom}{Isom}
\dmo{\conf}{Conf}
\dmo{\stab}{Stab}\dmo{\Jac}{Jac }
\dmo{\diam}{diam}\dmo{\fix}{Fixed}\dmo{\Fix}{Fix}
\dmo{\injR}{injRad}\dmo{\Ad}{Ad}
\dmo{\esv}{ess-vol}\dmo{\out}{Out}\dmo{\Out}{Out}
\dmo{\nil}{Nil}\dmo{\sol}{Sol}
\dmo{\Div}{div}
\dmo{\SU}{SU}
\dmo{\SL}{SL}
\dmo{\SP}{SP}
\dmo{\Sp}{Sp}
\dmo{\rk}{rk}
\dmo{\rank}{rank}
\dmo{\psp}{PSp}\dmo{\psu}{PSU}
\dmo{\PU}{PU}\dmo{\pgl}{PGL}
\dmo{\Mod}{Mod}\dmo{\range}{Range}
\dmo{\eu}{eu}\dmo{\mi}{mi}
\dmo{\Log}{Log}\dmo{\supp}{supp}
\dmo{\maps}{Maps}\dmo{\Gr}{Gr}
\dmo{\Pin}{Pin}
\dmo{\Spin}{Spin}\dmo{\Str}{Str}
\dmo{\Sq}{Sq}\dmo{\Symp}{Symp}
\dmo{\pd}{PD}\dmo{\PD}{PD}\dmo{\sig}{Sig}
\dmo{\Set}{Set}\dmo{\Top}{Top}
\dmo{\ev}{ev}\dmo{\St}{St}
\dmo{\Pt}{Pt}\dmo{\pt}{pt}

\dmo{\colim}{colim }\dmo{\Pl}{PL}
\dmo{\String}{String}\dmo{\smear}{smear}

\dmo{\dev}{dev}
\dmo{\met}{Met}\dmo{\contact}{Contact}
\dmo{\teich}{Teich}\dmo{\Teich}{Teich}\dmo{\qi}{QI}
\dmo{\der}{Der}
\dmo{\cl}{Cliff}\dmo{\Cl}{Cl}
\dmo{\Pf}{Pf}
\dmo{\ch}{ch}\dmo{\diag}{diag}
\dmo{\grad}{grad}\dmo{\Char}{char}
\dmo{\spec}{Spec}\dmo{\Arg}{Arg}
\dmo{\rad}{rad}\dmo{\im}{Im}
\dmo{\Hom}{Hom}\dmo{\End}{End}
\dmo{\tr}{tr}\dmo{\id}{Id}
\dmo{\gl}{GL}
\dmo{\sym}{Sym}\dmo{\Sym}{Sym}
\dmo{\com}{Comm}
\dmo{\Lk}{Lk}
\dmo{\CAT}{CAT}
\dmo{\Rep}{Rep}
\dmo{\Conf}{Conf}
\dmo{\PConf}{PConf}
\dmo{\Push}{Push}
\dmo{\Cont}{Cont}
\dmo{\sm}{\setminus}
\dmo{\vn}{\varnothing}
\dmo{\disk}{\mathbb D}
\dmo{\Trd}{Trd}\dmo{\Mat}{Mat}
\dmo{\Riem}{Riem}
\dmo{\Diffn}{\Diff_0}\dmo{\diff}{diff}
\dmo{\Diff}{Diff}\dmo{\Diffplus}{Diff^{\scaleobj{0.8}{+}}}
\dmo{\homeo}{Homeo}
\dmo{\Homeo}{Homeo}\dmo{\Fr}{Fr}
\dmo{\rot}{rot}\dmo{\Emb}{Emb}
\dmo{\Ham}{Ham}\dmo{\Met}{Met}
\dmo{\Ein}{Ein}\dmo{\CP}{\co P}
\dmo{\Per}{Per}\dmo{\Ric}{Ric}
\newcommand{\C}{\mathbb C}\dmo{\Nrd}{Nrd}
\dmo{\Comp}{Comp}\dmo{\PSC}{PSC}
\dmo{\Cent}{Cent}\dmo{\Orb}{Orb}
\dmo{\aind}{a-ind}\dmo{\tind}{t-ind}
\dmo{\constant}{constant}
\dmo{\Td}{Td}
\dmo{\LMod}{LMod}
\dmo{\SMod}{SMod}
\dmo{\SDiff}{SDiff}
\dmo{\Br}{Br}
\dmo{\csch}{csch}
\dmo{\triv}{triv}
\dmo{\genus}{genus}
\dmo{\Homeq}{HomEq}
\dmo{\PP}{\mathbb{P}}
\dmo{\U}{U}
\dmo{\Gal}{Gal}
\dmo{\BDiff}{\wtil{\Diff}}
\dmo{\BAut}{\wtil{\Aut}}
\dmo{\Iso}{Iso}
\dmo{\codim}{codim}
\dmo{\II}{II}
\dmo{\I}{I}

\usepackage{enumerate,pdfpages,stmaryrd} 
\usepackage[margin=1in]{geometry}

\usepackage{tikz}
\usepackage{dsfont}
\usetikzlibrary{matrix}

\begin{document}

\title[Geometric cycles and characteristic classes of manifold bundles]{Geometric cycles and characteristic classes of\\ manifold bundles}

\author{Bena Tshishiku,\\with an appendix by Manuel Krannich}
\address{Department of Mathematics, Harvard University, Cambridge, MA 02138} \email{tshishikub@gmail.com}
\address{Centre for Mathematical Sciences, University of Cambridge, Cambridge CB3 0WB, UK} \email{krannich@dpmms.cam.ac.uk}


\date{\today}


\begin{abstract} 
We produce new cohomology for non-uniform arithmetic lattices $\Ga<\SO(p,q)$ using a technique of Millson--Raghunathan. From this, we obtain new characteristic classes of manifold bundles with fiber a closed $4k$-dimensional manifold $M$ with indefinite intersection form of signature $(p,q)$. These classes are defined on finite covers of $B\Diff(M)$ and are shown to be nontrivial for $M=\#_g(S^{2k}\ti S^{2k})$. In this case, the classes produced live in degree $g$ and are independent from the algebra generated by the stable (i.e.\ MMM) classes. We also give an application to bundles with fiber a K3 surface.
\end{abstract}

\maketitle


\section{Introduction}

The starting point of this paper is the following new result about the cohomology of certain arithmetic groups $\Ga<\SO(p,q)$.

\begin{thm}\label{thm:cycle}
Fix $1\le p\le q$ such that $p+q\ge3$. Let $\Lam\sbs\R^{p+q}$ be a lattice with an integral bilinear form of signature $(p,q)$. Consider the group $\SO(\Lambda)$ of automorphisms of $\Lambda$ with determinant $1$. For every $n\ge1$, there is a finite-index subgroup $\Ga<\SO(\Lam)$ so that $\dim H^p(\Ga;\Q)\ge n$. 
\end{thm}

A lattice $\Lam$ as in Theorem \ref{thm:cycle} is determined up to isomorphism by its signature $(p,q)$ and its parity (even or odd) \cite[Ch.\ II, \S4]{milnor-husemoller}. The group $\SO(\Lam)$ is a nonuniform lattice in $\SO(p,q)\cong\SO(\Lam\ot_\Z\R)$ and has $\Q$-rank $p$. (Note we are using ``lattice" in two different ways. This should not cause any confusion.) 

Constructing nonzero elements of $H^*(\Ga;\Q)$ is a classical important problem in the theory of arithmetic groups. Our primary interest is to use Theorem \ref{thm:cycle} to produce new characteristic classes for certain fiber bundles. Recall that a characteristic class for fiber bundles with structure group $G$ is an element of $H^*(BG)$, where $BG$ is the classifying space of $G$. Our first application is as follows.

\begin{cor}\label{cor:main}
Fix $g\ge2$ and fix $n$ even so that $2n\ge g+4$. Let $W_g^{2n}=\#_g(S^{n}\ti S^{n})$, and denote the group of orientation-preserving diffeomorphisms of $W_g$ by $\Diff(W_g)$. For every $m\ge1$, there is a finite-index subgroup $\Diff^\Ga(W_g)<\Diff(W_g)$ so that $\dim H^g(B\Diff^\Ga(W_g);\Q)\ge m$.
\end{cor} 

Before discussing further applications, we make several remarks. 

The cohomology produced in Theorem \ref{thm:cycle} is new. Millson--Raghunathan \cite{millson-raghunathan} produce \emph{uniform} $\Ga<\SO(p,q)$ with $1\le p\le q$ and $p$ even such that $H^p(\Ga;\Q)\neq0$. Note in particular that \cite[pg.\ 103]{millson-raghunathan} requires $p$ to be even, whereas in Theorem \ref{thm:cycle} works for any $p$. This ``improvement" is obtained using some ideas from more recent work of Avramidi--Nguyen-Phan \cite{avramidi-phan}.

The particular finite-index subgroup $\Ga<\SO(\Lam)$ in Theorem \ref{thm:cycle} is somewhat difficult to make precise by the nature of the construction. The subgroup $\Diff^\Ga(W_g)<\Diff(W_g)$ appearing in Corollary \ref{cor:main} is defined as the preimage of a subgroup $\Ga<\SO(\Lam)$ under a homomorphism $\Diff(W_g)\ra\SO(\Lam)$, where $\Lam$ is $H_n(W_g^{2n})$ with its intersection form.

When $p=q$, then if $\Ga<\SO(\Lam)$ is finite index and $i\le p-2$, then $H^i(\Ga;\Q)\cong H^i(\SO(\Lam);\Q)$. This follows from Borel's stability theorem (see \cite{borel_cohoarith}, and see \cite{tshishiku-borel-range} for the stated range). Moreover, in this range, the cohomology ring $H^*(\SO(\Lam);\Q)$ is a polynomial ring with one generator in each degree $4k>0$. Theorem \ref{thm:cycle} shows that the stable range given in Borel's theorem is nearly sharp in this case. A similar observation for uniform lattices in $\SO(p,q)$ is mentioned in \cite{millson-raghunathan}.

Another way to state Theorem \ref{thm:cycle} is that the $p$-th \emph{virtual Betti number} of $\SO(\Lam)$ is infinite. If $\SO(\Lam)<\SO(p,q)$ was uniform, then to prove Theorem \ref{thm:cycle} it would suffice to prove that there exists $\Ga<\SO(\Lam)$ such that $H^p(\Ga;\Q)\neq0$. Then since $\SO(\Lam)$ has a large commensurator one can produce many linearly independent classes in further finite-index subgroups by an argument that appears in \cite{venkataramana-betti}. For a non-uniform lattice, the same argument works, but only up to the range in Borel's stability theorem. In our case, that range is less than $p$, so that argument cannot be used toward proving Theorem \ref{thm:cycle}. Our approach to showing $\dim H^p(\Ga;\Q)$ can be made large is along the lines of Avramidi--Nguyen-Phan \cite[Thm.\ 1.2]{avramidi-phan}, but the argument is different.

For $\Ga<\SO(\Lambda)$, a class $c\in H^*(\Ga;\Q)\cong H^*(B\Ga;\Q)$ can be viewed as a characteristic class for vector bundles $W\ra B$ with structure group $\Ga<\gl_{p+q}(\R)$. We explain what the characteristic classes produced in Theorem \ref{thm:cycle} measure using obstruction theory in \S\ref{sec:interpret}. This gives a new perspective on the Millson--Raghunathan construction. It also provides an interpretation for the classes in Corollary \ref{cor:main}. This becomes relevant in our application to bundles with fiber a K3 surface; see \S\ref{sec:K3}. 

The cohomology produced in Corollary \ref{cor:main} is new. The previously known classes in $H^*(B\Diff^\Ga(W_g);\Q)$ are the \emph{stable classes} (also known as \emph{tautological} or \emph{generalized Miller--Morita--Mumford} classes). For $B\Diff(W_g)$, the stable classes account for all of the cohomology in low degree $*\ll g$ \cite{GRW-homological-stability-1,GRW-homological-stability-2,GRW-stable-moduli}. When $g$ is odd, the classes we produce live in odd degree, whereas the stable classes all have even degree, so our classes are not in the algebra generated by the stable classes. A similar statement can be made for $g$ even. The nontriviality of these classes is detected by a nontrivial bundles, so Corollary \ref{cor:main} gives a new way to produce topologically nontrivial bundles $W_g\ra E\ra B^g$. 

Corollary \ref{cor:main} illustrates that the \emph{unstable} cohomology of arithmetic groups is a source of cohomology of $B\Diff(W_g)$. This phenomenon is largely unexplored; see also Corollary \ref{appendix:corollary:unstableclasses} in the appendix.

\vspace{.1in}
{\bf About the proof of Theorem \ref{thm:cycle}.} The cohomology classes in Theorem \ref{thm:cycle} are produced using geometric cycles in locally symmetric spaces. Let $X=\SO(p,q)/K$ be the symmetric space associated to $\SO(p,q)$, and let $Y=\Ga\bs X$ be the locally symmetric space for $\Ga<\SO(\Lam)$. There is an isomorphism $H^*(\Ga;\Q)\cong H^*(Y;\Q)$. Each class $c\in H^p(\Ga;\Q)$ we produce is Poincar\'e dual to a cycle $[Z]\in H_{pq-p}^{\text{cl}}(Y;\Z)$ in Borel--Moore homology with closed supports, and $[Z]$ is represented by a totally-geodesic, properly-embedded oriented submanifold $Z\sbs Y$. To show $[Z]\neq0$, we find a compact, totally-geodesic oriented submanifold $Z'\sbs Y$ of dimension $p$ so that the intersection number $[Z]\cdot[Z']\in H_0(Y;\Z)\cong\Z$ is nonzero. 

The cycles $[Z]$ and $[Z']$ are often called \emph{geometric cycles}. The idea of finding nontrivial homology of a locally symmetric space/arithmetic group by finding a pair of geometric cycles with $[Z]\cdot[Z']\neq0$ goes back to Millson \cite{millson}; see also \cite{millson-raghunathan, rohlfs-schwermer, lee-schwermer,avramidi-phan}. In each of these works, the locally symmetric space $Y=\Ga\bs X$ is either compact, or the lattice $\Ga$ is commensurable to $\Sl_n(\Z)$. The spaces $\Ga\bs\SO(p,q)/K$ we are interested in do not fall into either of these categories. Theorem \ref{thm:cycle} extends the known results to this case. 

In our argument $Z'\sbs Y$ is the quotient of a maximal periodic flat in $X$. Theorem \ref{thm:cycle} gives a partial answer to a question of Avramidi--Nguyen-Phan \cite[\S9]{avramidi-phan}. 

\vspace{.1in}
{\bf Characteristic class interpretation.} For the application to manifold bundles, the element $[Z]\in H_{pq-p}^{\text{cl}}(Y;\Q)\cong H^p(B\Ga;\Q)$ described above is not of any particular use as an abstract cohomology class. For this reason, one wants a bundle-theoretic construction of $[Z]$ as a characteristic class. 

Fix a lattice $\Lam\sbs\R^{p+q}$ as in Theorem \ref{thm:cycle}.  For a CW complex $B$, a map $B\ra B\SO(\Lam)$ defines a vector bundle $\R^{p+q}\ra W\ra B$ with a fiberwise lattice $\bsym\Lam$ and a fiberwise bilinear form $\bsym\be$ of signature $(p,q)$. We extend the structure group to $\SO(p,q)>\SO(\Lam)$ (this amounts to forgetting $\bsym\Lam$ but remembering $\bsym\be$), and then consider the different ways to reduce the structure group from $\SO(p,q)$ to its maximal compact subgroup. Each choice of reduction corresponds to a choice of a rank-$p$ subbundle $U\subset W$ on which $\bsym\be$ is positive definite. From this setup, we build a characteristic class $c$ that measures the difficulty of choosing $U\sbs W$ in a way that is ``compatible with $\bsym\Lam$." We make this precise in \S\ref{sec:interpret} using classical  obstruction theory, and we show that $c\in H^*(B\Ga;\Q)$ is dual to a geometric cycle $[Z]$.

\vspace{.1in}
{\bf Applications to manifold bundles.} Let $M^{4k}$ be a manifold, and let $\Lam_M$ denote the lattice $H_{2k}(M;\Z)/\text{torsion}$ with its intersection form. Given an $M$-bundle $\pi: E\ra B$, one can build a vector bundle $W\ra B$ by replacing each fiber $M_b:=\pi^{-1}(b)$ with its homology $H_{2k}(M_b;\R)$. On the level of classifying spaces, this corresponds to the map
\[\alpha:B\Diff(M)\ra B\oO(\Lam_M)\] 
induced by the action $\alpha: \Diff(M)\ra\oO(\Lam_M)$ of the group of oriention-preserving diffeomorphisms $\Diff(M)$ on $\Lam_M$ by automorphisms with determinant $\pm1$. For $\Ga<\SO(\Lam_M)$, we define $\Diff^\Ga(M)=\alpha^{-1}(\Ga)$. If $\Ga<\SO(\Lam_M)$ is finite index, then $\Diff^\Ga(M)<\Diff(M)$ is also finite index. In this case, note that any $M$ bundle $E\ra B$ has structure group reducing to $B\Diff^\Ga(M)$ after passing to a finite cover of $B$. 

To apply Theorem \ref{thm:cycle} to manifold bundles, we are interested in the homomorphism
\begin{equation}\label{eqn:coho}\al^*:H^*(B\Ga;\Q)\ra H^*\big(B\Diff^\Ga(M);\Q\big).\end{equation}

{\it Application to $W_g^{4k}=\#_g(S^{2k}\ti S^{2k})$.} When $M=W_g^{4k}$, information about $\alpha^*$ can be obtained using work of Berglund--Madsen \cite{berglund-madsen-rational}. This is explained in the appendix, written by Manuel Krannich, which studies $\alpha^*$ for the more general class of manifolds $W_g^{2n}=\#_g(S^n\times S^n)$ with $n\ge3$.  It is shown that $\alpha^*$ is injective in a range of degrees growing with $n$; see Theorem \ref{appendix:theorem}. Corollary \ref{cor:main} follows immediately from Theorems \ref{thm:cycle} and \ref{appendix:theorem}. As a further consequence of Theorem \ref{appendix:theorem}, the appendix produces unstable classes in the rational cohomology of $B\Diff(W_g^{2n})$ for $n\ge4$; see Corollary \ref{appendix:corollary:unstableclasses}.

We remark that the homomorphism $H^*(B\oO(\Lam_M);\Q)\ra H^*(B\Diff(M);\Q)$ can be completely understood in the stable range using index theory. Morita \cite{morita-ccs} showed this for $M$ a surface; see also \cite{EbertRW}. The techniques used to study $\alpha^*$ outside the stable range rely on surgery theory, Morlet's lemma of disjunction, and rational homotopy theory; see the appendix. 

\vspace{.1in}
{\it Application to K3 surfaces.} Let $M^4$ be a manifold diffeomorphic to a K3 surface. In this case $\SO(\Lam_M)$ is a lattice in $\SO(3,19)$, and by Theorem \ref{thm:cycle}, we can find finite-index $\Ga<\SO(\Lam_M)$ and a nonzero cycle $z\in H_3(B\Ga;\Q)$. Using the global Torelli theorem, we conclude that $z$ is in the image of $H_3(B\pi_0\Diff^\Ga(M);\Q)\ra H_3(B\Ga;\Q)$. We are not able to determine if $z$ is in the image of $\al_*:H_*(B\Diff^\Ga(M);\Q)\ra H_*(B\Ga;\Q)$, but we relate this problem to another problem of interest. Specifically, we give an example $z\neq0\in H_3(B\Ga;\Q)$ so that if $z$ is in the image $\al_*$, then there exists a K3-surface bundle over a 3-manifold that does not admit a fiberwise Einstein metric. This should be contrasted with a theorem of Donaldson \cite[Cor.\ 6.3]{donaldson} that says that every K3 bundle over $S^1$ admits a fiberwise Einstein metric; when the base has dimension at least 2, the corresponding statement seems to be unknown. 

\vspace{.1in}
{\it Odd-dimensional manifolds.} With the methods of this paper, we can also produce characteristic classes for $M$ bundles when $\dim M$ is odd. A sample application to $M=\#_3(S^d\ti S^{d+1})$ is discussed in \S\ref{sec:app}.

\vspace{.1in} 
{\bf Section outline.} In \S\ref{sec:geocycle} we recall the general method of constructing homology of arithmetic groups using geometric cycles. In \S\ref{sec:thm1} we apply that method to $\SO(\Lam)$ and prove Theorem \ref{thm:cycle}. In \S\ref{sec:interpret} we explain how to view geometric cycles as characteristic classes. Finally, \S\ref{sec:app} and the Appendix contain the applications to manifold bundles.  

\vspace{.1in} 
{\bf Acknowledgements.} BT would like to thank S.\ Kupers for suggesting Corollary \ref{cor:main}. He thanks M.\ Grey, P.\ Kronheimer, and T.\ Nguyen-Phan for useful conversations and G.\ Avramidi and J.\ Schwermer for helpful email correspondence. Thanks to B.\ Farb for comments on a draft and for his continued support. He is grateful to O.\ Randal-Williams for pointing out a pair of errors in a previous draft of this paper. Finally, thanks to M.\ Krannich for writing the appendix, which corrects and improves an argument that appeared in an earlier draft of this paper. 

MK would like to thank B.\ Tshishiku for offering him to write the appendix and O.\ Randal-Williams for comments on a draft of it. He was supported by the European Research Council (ERC) under the European Union’s Horizon 2020 research and innovation programme (grant agreement No. 756444).

Finally, BT thanks the referee who pointed out a mistake in the proof of Theorem 1. More generally, thanks to the referees for carefully reading the paper and offering many comments and corrections, which have been invaluable.

\section{Homology of arithmetic groups: geometric cycles}\label{sec:geocycle}

This section provides the setup for proving Theorem \ref{thm:cycle}. We summarize the general strategy to produce geometric cycles in the homology of arithmetic groups/locally symmetric manifolds. We follow \cite{millson-raghunathan} and \cite{schwermer-survey} and refer the reader to these sources for further details. In \S\ref{sec:thm1} we will apply the material of this section to the specific case of interest $\Ga<\SO(\Lam)$.

\subsection{Geometric cycles: the general strategy.}\label{sec:geometric-cycles} 

Fix an algebraic $\Q$-group $G$ such that $G(\R)$ is a semisimple Lie group without compact factors. We are interested in finding some nontrivial homology of a finite-index subgroup $\Ga<G(\Z)$. The exact subgroup $\Ga$ will be (unfortunately) out of our control, and at several points we will replace $\Ga$ with a further finite-index subgroup (without changing the notation) to ensure that some geometric fact is true. 

We begin by describing the locally symmetric model for $B\Ga$. Choose a maximal compact subgroup $K<G(\R)$, and define $X=G(\R)/K$. The manifold $X$ is contractible and admits a $G(\R)$-invariant Riemannian metric of nonpositive curvature. Since $\Ga<G(\R)$ is discrete, it acts properly discontinuously on $X$ and each point-stabilizer in $\Ga$ is finite. We can replace $\Ga$ by a torsion-free, finite-index subgroup \cite[Thm.\ 4.8.2]{morris}, so then $\Ga$ acts freely on $X$. Then $Y=\Ga\bs X$ is a model for $B\Ga$. The manifold $Y$ may be noncompact, but it has finite volume (because arithmetic subgroups are lattices \cite[Thm.\ 1.3.9]{morris}). 

The manifold $Y$ has an abundance of totally geodesic immersed submanifolds. Let $G_1<G$ be a subgroup and take $h\in G(\R)$ so that $K_1:=G_1(\R)\cap (hKh^{-1})$ is a maximal compact subgroup of $G_1(\R)$. The image of the orbit map $G_1(\R)\ni g\mapsto ghK\in X$ is totally geodesic submanifold $G_1(\R)/K_1\cong X_1\sbs X$. If $G_1$ is a $\Q$-subgroup, then $\Ga_1=G_1\cap\Ga$ has finite index in $G_1(\Z)$, and the natural map $j_1:Y_1=\Ga_1\bs X_1\ra Y$ is a proper, totally geodesic immersion \cite[\S6]{schwermer-survey}. 

With this setup, we are ready to discuss the general strategy for producing ``geometric cycles" in the homology of $Y$. Let $o=eK$ be the basepoint of $X$. First choose $G_1,G_2<G$ so that 
\begin{enumerate}
\item[($\dagger$)] $X_1$ and $X_2$ have complementary dimension $d_1+d_2=\dim X$, the intersection $X_1\cap X_2=\{o\}$ is transverse, and $Y_1$ is compact.
\end{enumerate}
Then choose $\Ga<G(\Z)$ so that 
\begin{enumerate}
\item[($\ddagger$)] the quotients $Y,Y_1,Y_2$ are oriented manifolds, the maps $j_1,j_2$ are embeddings, and the intersection $Y_1\cap Y_2$ is transverse and every intersection has positive sign. 
\end{enumerate} 

Given ($\dagger$) and ($\ddagger$), the submanifolds $Y_i\sbs Y$ determine classes $[Y_1]\in H_{d_1}(Y;\Q)$ and $[Y_2]\in H^{\text{cl}}_{d_2}(Y;\Q)$ in homology and homology with closed supports, and the algebraic intersection $[Y_1]\cdot[Y_2]$ is nonzero, so $[Y_1]\neq0$ in $H_{d_1}(Y;\Q)\cong H_{d_1}(\Ga;\Q)$, and $[Y_2]\neq0$ in $H_{d_2}^{\text{cl}}(Y;\Q)\cong H^{d_1}(Y;\Q)\cong H^{d_1}(\Ga;\Q)$. This is explained in more detail in \cite{schwermer-survey}. 

The general strategy does not always work. Indeed, it is not always possible to achieve ($\dagger$). One problem is that a totally geodesic subspace $X_1\sbs X$ need not admit a subspace of complementary dimension. (It is shown in \cite[Theorem 1.1]{millson-raghunathan} that such a complement exists if $X_1=X^\si$ is the fixed set of an involutive isometry.) Another problem is that if $Y$ is non-compact, then there is no reason $Y_1$ or $Y_2$ need be compact in general. Nevertheless, in some special cases, one can find $G_1,G_2$ so that ($\dagger$) is satisfied. We will see this when $G$ is an indefinite orthogonal group in the next section. 

Now we address the difficulty with ($\ddagger$). There is a general theorem that ensures the first two clauses of ($\ddagger$). (Ensuring that $Y$ is oriented is easy, but ensuring $Y_1$ and $Y_2$ are oriented is already nontrivial.)  

\begin{thm}[\cite{schwermer-survey}, Theorem D]\label{thm:schwermer}
Let $G$ be a connected semisimple algebraic $\Q$-group, let $G_1<G$ be a connected reductive $\Q$-subgroup, and let $\Ga<G(\Q)$ be an arithmetic subgroup. Then after replacing $\Ga$ by a finite-index subgroup, the map $j_1:Y_1\ra Y$ is a proper, injective, closed embedding, and each component of the image is an orientable, totally geodesic submanifold of $Y$. 
\end{thm}

Given Theorem \ref{thm:schwermer}, the remaining difficulty is showing that, after replacing $\Ga$ with a subgroup of large index, $Y_1\cap Y_2$ is a finite set of points and the intersection number at each point is $+1$. We explain how to approach this problem in the next subsection. 

\subsection{Intersections and double cosets.}\label{sec:intersect} Assume that $G_1,G_2<G$ satisfy ($\dagger$) and that $\Ga<G(\Z)$ is torsion-free and the associated manifolds $Y_1$ and $Y_2$ are oriented, embedded submanifolds of $Y$. As explained in \cite{millson-raghunathan} (to be reviewed below), the components of $Y_1\cap Y_2$ can be identified with a certain subset $\Om$ of the double coset space $\Ga_2\bs\Ga/\Ga_1$. Choosing coset representatives $I(\Ga)\sbs\Ga$ for $\Om$, the sign of the intersection corresponding to $\ga\in I(\Ga)$ is determined by the double coset of $\ga$ in $G_2^+(\R)\bs G(\R)/G_1^+(\R)$, where $G_i^+(\R)<G_i(\R)$ is the subgroup that preserves orientation on $X_i$ and on $X$ (these conditions are not automatic if $G_i(\R)$ and $G(\R)$ are not connected). In particular, if $\ga$ can be written $\ga=g_2g_1$ with $g_i\in G_i^+(\R)$, then the corresponding intersection is positive \cite[Cor.\ to Lem. 2.5]{millson-raghunathan}. One wants to show that if $\Ga<G(\Z)$ is a subgroup of large index, then every $\ga\in I(\Ga)$ belongs to one of the double cosets in $G_2^+(\R)\bs G(\R)/G_1^+(\R)$ that correspond to a positive intersection number. 

{\bf Components of $Y_1\cap Y_2$.} First we describe the set $\pi_0(Y_1\cap Y_2)$ in terms of the double coset space $\Ga_2\bs\Ga/\Ga_1$. Denote the projection $\Pi:X\ra Y$. There is a bijection between $\Pi^{-1}(Y_1\cap Y_2)$ and 
\[T:=\{(\ga,x_1,x_2): \ga x_1=x_2\}\sbs \Ga\ti X_1\ti X_2\] (an intersection downstairs is covered by an intersection upstairs, and we can translate by $\Ga$ so that the intersection happens on $X_2$). The set $T$ has an action of $\Ga_2\ti\Ga_1$ given by 
\[(\al_2,\al_1).(\ga,x_1,x_2)=(\al_2\ga\al_1^{-1},\al_1x_1,\al_2x_2).\]

{\it Claim.} $Y_1\cap Y_2\cong T/(\Ga_2\ti\Ga_1)$. 

To prove the claim, one shows that if $(\ga,x_1,x_2)$ and $(\ga',x_1',x_2')$ are in $T$, then $\Pi(x_1)=\Pi(x_1')$ if and only if there exists $\al_i\in\Ga_i$ so that $(\ga',x_1',x_2')=(\al_2,\al_1).(\ga,x_1,x_2)$. The ``if" direction is obvious. For the ``only if" direction, one uses the fact that $Y_i$ is embedded in $Y$, which implies that if $\ga\in\Ga$ and $\ga X_i\cap X_i\neq\vn$, then $\ga\in \Ga_i$. More details can be found in \cite[\S2]{millson-raghunathan}. 

A similar argument shows that if $(\ga,x_1,x_2)$ and $(\ga',x_1',x_2')$ are in $T$, then $\Pi(x_1)$ and $\Pi(x_1')$ are in the same component of $Y_1\cap Y_2$ if and only if $\ga$ and $\ga'$ lie in the same double coset $\Ga_2\bs\Ga/\Ga_1$. See \cite{millson-raghunathan} Lemma 2.3 and Proposition 2.3. In other words, $\pi_0(Y_1\cap Y_2)$ is in bijection with  
\[\Om:=\{\Ga_2\ga\Ga_1: \ga X_1\cap X_2\neq\vn \}\sbs  \Ga_2\bs\Ga/\Ga_1.\]
Note that $\Om\cong\pi_0(Y_1\cap Y_2)$ is finite because $Y_1\cap Y_2\sbs Y_1$ is a submanifold and $Y_1$ is compact. 

{\bf Sign of the intersection.} Next we explain, for each $y\in \pi_0(Y_1\cap Y_2)$, whether the intersection is positive, negative, or degenerate. Fix a set of coset representatives $I(\Ga)\sbs\Ga$ for elements of $\Om$, and assume that $\ga_1=\id$ represents $\Ga_2\Ga_1$. 

Note that for $\ga\in I(\Ga)$ we can write $\ga=a_2 k a_1^{-1}$, where $a_i\in G_i^+(\R)$ and $k\in K$. This is because $\ga X_1\cap X_2\neq\vn$ implies that there exists $x_i\in X_i$ so that $\ga x_1=x_2$. Since $G_i^+(\R)$ acts transitively on $X_i$, we can choose $a_i\in G_i^+(\R)$ so that $a_i (o)=x_i$, where $o=eK$ is the basepoint of $X=G(\R)/K$ (and is also the intersection of $X_1$ and $X_2$).  Then $a_2^{-1}\ga a_1(o)=o$, which means $a_2^{-1}\ga a_1=k$ for some $k\in K$. 

Since $a_i$ preserves orientation on $X_i$ and $X$, the sign of the intersection $\ga X_1\cap X_2$ is determined by the action of $k$ on $T_oX$: 

\begin{lem}\label{lem:sign} 
Let $e_1,\ld,e_p\in T_oX_1$ and $e_{p+1},\ld,e_n\in T_oX_2$ be positively oriented bases. Define $\ep(\ga)$ by 
\begin{equation}\label{eqn:ep}k(e_1)\we \cd\we k(e_p)\we e_{p+1}\we\cd\we e_{n}=\ep(\ga)\>e_1\we\cd\we e_{n}.\end{equation}
\begin{enumerate}
\item[(a)] The intersection corresponding to $\ga\in I(\Ga)$ is positive, negative, or degenerate according to whether $\ep(\ga)$ is positive, negative, or zero.
\item[(b)] If $\ga,\ga'\in I(\Ga)$ lie in the same coset $G_2^+(\R)\bs G(\R)/ G_1^+(\R)$, then $\ep(\ga)=\ep(\ga')$. 
\end{enumerate} \end{lem}

\begin{proof}
We begin with (a). The basis $(a_1e_1,\ld,a_1e_p)$ is positively oriented in $T_{x_1}X_1$, and the basis $(a_2e_{p+1},\ld,a_2e_n)$ is positively oriented in $T_{x_2}X_2$. We want to know if 
\[\big(\ga(a_1e_1),\ld,\ga(a_1e_p),a_2e_{p+1},\ld,a_2e_n\big)\] is positively oriented in $T_{x_2}X$. Since $a_2$ preserves orientation on $X$ and $\ga=a_2ka_1^{-1}$, the orientation of this $n$-tuple is the same as the orientation of $(ke_1,\ld,ke_p,e_{p+1},\ld,e_n)$. See also \cite[Prop.\ 2.3]{millson-raghunathan}. 

For (b), assume $\ga,\ga'\in I(\Ga)$ and $\ga'=h_2\ga h_1$ for some $h_i\in G_2^+(\R)$. There are two cases: the intersection $\ga X_1\cap X_2$ is either degenerate or not. If $\ga X_1\cap X_2$ is degenerate (i.e.\ has dimension at least 1), then the same is true for $h_2\ga h_1=\ga'$, so $\ep(\ga)=0=\ep(\ga')$. If $\ga X_1\cap X_2$ and hence also $\ga'X_1\cap X_2$ are non-degenerate, then there exists a unique $x_1,x_1'\in X_1$ and $x_2,x_2'\in X_2$ so that $\ga x_1=x_2$ and $\ga'x_1'=x_2'$. Since $\ga'=h_2\ga h_1$, it follows that $x_1=h_1x_1'$ and $x_2=h_2^{-1}x_2'$. Then if $\ga=a_2ka_1^{-1}$ where $a_i(0)=x_i$, then we have $\ga'=h_2a_2ka_1^{-1}h_1$, and $h_1^{-1}a_1(0)=h_1^{-1}(x_1)=x_1'$ and $h_2a_2(0)=h_2(x_2)=x_2'$, which implies that both $\ep(\ga)$ and $\ep(\ga')$ is computed from the action of $k$ as in (\ref{eqn:ep}), so $\ep(\ga)=\ep(\ga')$. See also \cite[Lem.\ 2.5]{millson-raghunathan}. 
\end{proof}

We note that (b) implies that the sign of the intersection doesn't depend on the choice of representative $\ga\in I(\Ga)$ as long as $\Ga_i\sbs G_i^+(\R)$, which can be arranged by passing to a finite-index subgroup (c.f.\ Theorem \ref{thm:schwermer}). In this case we have the following formula 
\[[Y_1]\cdot[Y_2]=\sum_{\ga\in I(\Ga)}\ep(\ga).\]

Note in addition that (b) implies that if $\ga\in G_2^+(\R) G_1^+(\R)$, then $\ep(\ga)=+1$. We can see this latter fact directly as follows: if $\ga=g_2g_1\in G_2^+(\R)G_1^+(\R)$, then 
\begin{equation}\label{eqn:samesign}\ga X_1\cap X_2=(g_2g_1X_1)\cap X_2=(g_2X_1)\cap X_2=g_2(X_1\cap  g_2^{-1}X_2)=g_2(X_1\cap X_2),\end{equation} and since $g_2$ preserves orientation on $X$ and $X_2$, this implies that the sign of the intersection $\ga X_1\cap X_2$ is equal to the sign of the intersection $X_1\cap X_2$, which is positive by assumption.

{\bf Writing $\ga\in I(\Ga)$ in $G_2^+(\R)G_1^+(\R)$.} As mentioned above, one way to show an intersection $\ga X_1\cap X_2$ is positive is to show that $\ga\in G_2^+(\R)G_1^+(\R)$. The next two propositions are partial progress to writing $\ga$ in this form (which is not always possible in general, but our goal will be to show it can be ensured by passing to a finite-index subgroup a given $\Ga$). The following proposition follows from the argument of \cite[Theorem 3.1]{millson-raghunathan}; see also \cite[Lem.\ 2.6]{for}. 

\begin{prop}\label{prop:complexproduct}
Let $G$ be an connected, reductive algebraic $\Q$-group with $K<G(\R)$ a maximal compact subgroup. Fix an arithmetic subgroup $\Ga<G(\Z)$, and let $G_1,G_2<G$ be connected, reductive $\Q$-subgroups. There exists a finite-index subgroup $\Ga'<\Ga$ so that if $\ga\in \Ga'$ and $G_2(\R)\ga\cap KG_1(\R)\neq\vn$, then $\ga\in G_2(\C)G_1(\C)$. 
\end{prop}

From the above, we see that if $\ga X_1\cap X_2\neq\vn$, then $G_2(\R)\ga\cap KG_1(\R)\neq\vn$. By Proposition \ref{prop:complexproduct}, after replacing $\Ga$ by a finite-index subgroup, we can ensure that if $\ga\in I(\Ga)$, then $\ga\in G_2(\C)G_1(\C)$. We strengthen this with the following proposition. 

\begin{prop}\label{prop:rationalproduct}Fix $G,G_1,G_2$ and $\Ga$ as above. Assume that $G_1\cap G_2=\{\id\}$. For $\ga\in\Ga$, if $\ga\in G_2(\C) G_1(\C)$, then $\ga\in G_2(\Q)G_1(\Q)$.
\end{prop}

\begin{rmk}\label{rmk:disjoint}
Note that if $X_1\cap X_2=\vn$ and $\ga=g_2g_1\in G_2(\Q)G_1(\Q)$, then also $\ga X_1\cap X_2=\vn$ by the computation in (\ref{eqn:samesign}). Thus, as a consequence of Proposition \ref{prop:rationalproduct} and the above discussion, if $X_1$ and $X_2$ are disjoint, then there exists finite index $\Ga<G(\Z)$ so that the quotient submanifolds $Y_1,Y_2$ are disjoint in $Y$. This observation will be used in \S\ref{sec:arrange} to produce linearly independent cycles. 
\end{rmk}

\begin{proof}[Proof of Proposition \ref{prop:rationalproduct}]
We are given $\ga=h_2h_1$ with $h_i\in G_i(\C)$. To show that $h_i\in G_i(\Q)$ we show that $h_i\si(h_i)=e$ for every $\si\in\Gal(\C/\Q)$. Given $\si\in\Gal(\C/\Q)$, since $\si(\ga)=\ga$, we conclude that $h_1\si(h_1)^{-1}=h_2^{-1}\si(h_2)$. This equality implies $h_1\si(h_1)^{-1}\in G_1(\C)\cap G_2(\C)$, which is trivial assumption. Hence $h_1=\si(h_1)$ and similarly $h_2=\si(h_2)$.
\end{proof}

In summary, we have shown that $\ga\in I(\Ga)$ can be expressed as $\ga=g_2g_1$ with $g_i\in G_i(\Q)$. 

{\bf Eliminating intersections.} Assume $\Ga$ satisfies the hypothesis of Proposition \ref{prop:complexproduct} and hence the conclusion of Proposition \ref{prop:rationalproduct}. Write $I^+(\Ga)$ for the subset of $\ga\in I(\Ga)$ for which the intersection $\ga X_1\cap X_2$ is positive. We finish this section with two easy lemmas that we will use to prove Theorem \ref{thm:cycle} in \S\ref{sec:thm1}. 

\begin{lem}\label{lem:endgame}
Take $\Ga$ as in the preceding paragraph. Let $\Ga'<\Ga$ be a finite-index subgroup. If $\Ga'\cap\Ga_2\ga\Ga_1=\varnothing$ for every $\ga\in I(\Ga)\bs I^+(\Ga)$, then for every $\ga'\in\Ga'$ for which $\ga'X_1\cap X_2\neq\vn$, the sign of the intersection is positive. 
\end{lem}

\begin{proof}
Denoting $\Ga_i'=\Ga'\cap\Ga_i$, there is a map $\Om'=\Ga_2'\bs\Ga'/\Ga_1'\ra\Ga_2\bs\Ga/\Ga_1=\Om$ induced by the inclusion $\Ga'\hra\Ga$. Fixing $\ga'$ as in the statement, by assumption $\Ga_2\ga'\Ga_1=\Ga_2\ga\Ga_1$, where $\ga\in I^+(\Ga)$. By Lemma \ref{lem:sign}, the sign of the intersection $\ga'X_1\cap X_2$ is the sign of $\ga X_1\cap X_2$, which is positive since $\ga\in I^+(\Ga)$. 
\end{proof}

\begin{lem}\label{lem:normal}
Suppose that $\Ga'\lhd\Ga$ is a normal subgroup. Fix $\ga\in\Ga$ and $\Ga_1,\Ga_2<\Ga$. Then $\Ga'\cap\Ga_2\ga\Ga_1\neq\vn$ if and only if $\Ga'\ga\cap\Ga_2\Ga_1\neq\vn$. 
\end{lem}

\begin{proof}
The proof is straightforward. If $\ga'=\ga_2\ga\ga_1$ with $\ga'\in\Ga'$ and $\ga_i\in\Ga_i$, then 
\[\ga_1^{-1}=(\ga')^{-1}\ga_2\ga=\ga_2(\ga_2^{-1}(\ga')^{-1}\ga_2)\ga.\]
Equivalently, $\ga_2^{-1}\ga_1^{-1}=(\ga_2^{-1}(\ga')^{-1}\ga_2)\ga$, which implies that $\Ga'\ga\cap\Ga_2\Ga_1\neq\vn$. The other direction is similar.  
\end{proof}

\section{Geometric cycles for $\Ga<\SO(p,q)$}\label{sec:thm1}

In this section we prove Theorem \ref{thm:cycle}. Fix $1\le p\le q$ and $\Lam\sbs\R^{p+q}$ as in the statement of the theorem. Let $B$ be the matrix for the bilinear form on $\Lam$ with respect to some basis, and consider the algebraic $\Q$-group 
\begin{equation}\label{eqn:G}G=\SO(B)=\{g\in\Sl_{p+q}(\C): g^t Bg=B\}.
\end{equation}
Setting $\Ga=G(\Z)\cong\SO(\Lam)$, we split the proof of the theorem into proving two statements: 
\begin{enumerate}
\item[(a)] Up to replacing $\Ga$ by a finite-index subgroup, $H_p(\Ga;\Q)$ is nonzero. 
\item[(b)] Given $n\ge1$, we can replace $\Ga$ by a finite-index subgroup so that $\dim H_p(\Ga;\Q)\ge n$. 
\end{enumerate} 
In \S\ref{sec:dagger}, we define algebraic groups $G_1,G_2$ so that $G_1(\R)\cong\SO(1,1)^p$ and $G_2(\R)\cong\SO(p,q-1)$, and we verify that the conditions of ($\dagger$) from \S\ref{sec:geocycle} can be satisfied for a good choice of $G_1,G_2<G$. In \S\ref{sec:ddagger}, we show that we can choose $\Ga<G(\Z)$ so that ($\ddagger$) is also satisfied. Together these prove (a). In \S\ref{sec:arrange} we prove (b) by showing how to produce many linearly independent flat cycles. 

\subsection{Choosing $G_1,G_2<G$.} \label{sec:dagger} 
Let $V=\Lam\ot_\Z\R$. 

We choose $G_1$ as a maximal $\R$-split torus contained in the centralizer of a \emph{hyper-regular} element $\tau\in G(\Z)$, in the sense of Prasad--Raghunathan \cite{prasad-raghunathan}. The group $G_1$ is defined and anisotropic over $\Q$ (i.e.\ $G_1$ does not contain any nontrivial $\Q$-split torus, which implies that $G_1(\R)/G_1(\Z)$ is compact). One can take $\tau$ to preserve a decomposition $V=U\oplus U^\perp$ defined over $\Q$ with $U^\perp$ negative definite of dimension $q-p$ and such that the action of $\tau$ on $U$ is irreducible and has $2p$ distinct real eigenvalues occurring in $(\la,1/\la)$ pairs (irreducible implies in particular that the eigenvectors of $\tau$ are not defined over $\Q$). See below for a concrete example. 

For $G_2$, we fix $\la\in\Lam$ with $\lambda\cdot\lambda<0$ and define $G_2\cong\SO(B')$, where $B'$ is the restriction of $B$ to $\la^\perp$. This group includes into $G(B)$ in an obvious way, acting trivially on $\langle\la\rangle$. 

Our groups have real points $G(\R)\cong\SO(p,q)$, $G_1(\R)\cong\SO(1,1)^p\cong (\R^\ti)^p$, and $G_2(\R)\cong\SO(p,q-1)$. The associated symmetric spaces $X$, $X_1$, $X_2$ have dimensions $pq$, $p$, $p(q-1)$, respectively.

{\bf Example.} We give an explicit example of the group $G_1$ for $\SO(2,2)$ by an ad hoc construction. Consider the number field $F=\Q(\al)$, where $\al$ is a root of $t^4-12t^3+23t^2-12t+1$. The elements $\al$ and $\fr{1}{7}(16\al^3-180\al^2+233\al-12)$ are units in $\ca O_F$ and act on $\ca O_F\cong\Z^4$ by the matrices 
\[\tau_1=\left(\begin{array}{rrrrr}-1&0&0&-7\\2&0&0&13\\-3&1&0&-22\\2&0&1&13
\end{array}\right)
\>\>\>\text{ and }\>\>\>\>\>
\tau_2=\left(\begin{array}{rrrr}-7&-16&-12&-9\\14&28&20&15\\-9&-17&-12&-10\\2&4&3&3
\end{array}\right).\]
These matrices preserve the bilinear form with matrix 
\[B=\left(\begin{array}{cccc}6&2&2&19\\2&0&1&12\\2&1&0&1\\19&12&1&0
\end{array}\right),
\]
which has signature (2,2) and is unimodular ($\det B=1$). 
Let $G_1<\SO(B)$ be the centralizer of $\tau_1$. Since $\tau_1$ has 4-distinct real eigenvalues (none with norm 1), $G_1(\R)\cong\SO(1,1)^2$ is a maximal torus in $\SO(B)(\R)$. Furthermore, $G_1$ is $\Q$-anisotropic because the eigenvectors of $\tau_1$ are not defined over $\Q$. Note that $G_1(\Z)<G_1(\R)$ is cocompact since it contains $\langle \tau_1,\tau_2\rangle\cong\Z^2$. This construction can also be used to give examples in $\SO(2,q)$ for $q>2$ by taking the direct sum with a negative definite form and extending by the identity. In this case, the $\R$-points of the centralizer of $\tau_1$ is $\SO(1,1)^p\ti\SO(q-p)$. 

{\bf Generic pairs $G_1,G_2$.} We return to the general setup $G_1,G_2<G=\SO(B)$. We want to choose $G_1,G_2$ so that $X_1\cap X_2=\{o\}$ is a single point. This leads us to define a notion of generic pairs $G_1,G_2$, which will be useful at various points. We say that $G_1,G_2$ are \emph{generic} if $G_1\cap G_2=\{\id\}$. We interpret this condition as an equality of algebraic groups from the ``functor of points" point-of-view, so that $G_1(F)\cap G_2(F)=\{\id\}$ for each field extension $F/\Q$. In particular, we will use this for $F=\Q,\R, \Q_p$. 

For example, $G_1(\R)\cap G_2(\R)=\{\id\}$ implies that the intersection $X_1\cap X_2$ is transverse, i.e.\ either a point or the empty set (it's empty if there is no maximal compact subgroup of $G(\R)$ that intersects both $G_1(\R)$ and $G_2(\R)$ in respective maximal compact subgroups). To see this implication, suppose that $\dim X_1\cap X_2\ge1$. Without loss of generality we assume $o\in X_1\cap X_2$. Then $\dim(T_oX_1\cap T_oX_2)\ge1$ implies that $\dim (\mf g_1\cap \mf g_2)\ge1$, where $\mf g_i$ is the Lie algebra of $G_i(\R)$. Finally, this implies that $\dim (G_1(\R)\cap G_2(\R))\ge1$. 

The following proposition gives a sufficient condition for $G_1,G_2$ to be generic. To state it, consider the action of $G(\Q)$ on $V_\Q:=\Lam\ot_\Z\Q$. The group $G_1(\Q)$ preserves a subspace $U$ of signature $(p,p)$ and acts trivially on $U^\perp$. The group $G_2(\Q)$ preserves $\la^\perp\sbs V_\Q$ and acts trivially on $\langle\la\rangle$. 

\begin{prop}\label{prop:generic}
Take $G_1,G_2<G$ and take $U,\langle\la\rangle\sbs V_\Q$ as in the preceding paragraph. If $\la\notin U^\perp$, then $G_1,G_2$ are generic. 
\end{prop}

\begin{proof}
First we set some notation. Let $\tau\in G(\Z)$ be the hyper-regular element used to define $G_1$. Let $V_\Q=U\oplus U^\perp$ the decomposition preserved by $\tau$, where $U^\perp$ is negative-definite of dimension $q-p$ and $\tau$ acts irreducibly on $U$. Let $\la\in\Lam$ be the vector used to define $G_2$. Observe that $G_2$ can be described in terms of a centralizer: let $\ep\in\gl_{p+q}(\Q)$ be the automorphism of $V_\Q$ that fixes $\la$ and acts as $-1$ on its orthogonal complement. Then $G_2$ is the subgroup of the centralizer of $\ep$ in $G(\Q)$ that acts trivially on $\langle\la\rangle$. 

For an extension $F/\Q$, we denote $C_\tau(F),C_\ep(F)$ the centralizers of $\tau$ and $\ep$ in the (vector) space of $(p+q)\ti(p+q)$ matrices over $F$. 

{\it Claim.} If $a\in C_\tau(\Q)\cap C_\ep(\Q)$, then $a$ acts as on $U$ by multiplication by a scalar $x\in\Q$. 

{\it Proof of Claim.} Since $\ep(\la)=\la$ and $a$ commutes with $\ep$, $\la$ is an eigenvector of $a$, i.e.\ $a(\la)=x\la$ for some $x\in\Q$. Write $\la=u+v\in U\oplus U^\perp$. We know $u\neq0$ because we're assuming $\la\notin U^\perp$. Also $a$ preserves $U^\perp$ because $\tau\rest{}{U^\perp}=\id$ and $a$ commutes with $\tau$. Then $a$ also preserves $U=(U^\perp)^\perp$. Then 
\[xu+xv=x(u+v)=a(u+v)=a(u)+a(v)\]
implies that $u$ and $v$ are both eigenvectors for $a$ with eigenvalue $x$. In addition 
\[a(\tau^i(u))=\tau^i(a(u))=x\tau^i(u)\]
so $\tau^i(u)$ is also an eigenvector for $a$ with eigenvalue $x$ for each $i\in\Z$. Since $u\in U$ and $\tau$ acts irreducibly on $U$, this implies that $a$ acts on $U$ by multiplication by $x$. This proves the claim.

Using the claim, it follows that if $F/\Q$ is an extension then any $a\in C_\tau(F)\cap C_\ep(F)$ acts on $U\ot F$ by a scalar $x\in F$. This is because the conditions $a\tau=\tau a$ and $a\ep=\ep a$ give a system of linear equations defined over $\Q$, so the set of solutions is described independently of the field. 

Now we finish the proof of the proposition by showing that $G_1(F)\cap G_2(F)=\{\id\}$ for any field extension $F/\Q$. Since $G_1(F)\cap G_2(F)\sbs C_\tau(F)\cap C_\ep(F)$, any $a\in G_1(F)\cap G_2(F)$ acts on $U\ot F$ by multiplication by a scalar $x\in F$. But since $a$ acts by an isometry, $x^2=1$ so $x=\pm1$. To show $a=\id$ we want to show $x=1$ (we already know that $a$ acts as $\id$ on $U^\perp$ since $a\in G_1(F)$). Since $a\in G_2(F)$, $a(\la)=\la$. Writing $\la=u+v$ as before, then
\[u+v=\la = a(\la)=a(u)+a(v)=xu+v.\] Since $u\neq0$, this implies $x=1$. This completes the proof. 
\end{proof}

Note that the hypothesis $\la\notin U^\perp$ is automatically satisfied when $p=q$ since then $U^\perp=0$. In addition, given $G_1$ and $G_2$, we can replace $\la$ with $\la'$ so that the rational lines are $\langle\la'\rangle$ and $\langle\la\rangle$ are arbitrarily close and $G_1,G_2'$ are generic. 

\subsection{Eliminating intersections}\label{sec:ddagger} 

In this section we start with $G_1,G_2$ generic with $X_1\cap X_2=\{o\}$ and with $\Ga<G(\Z)$ that preserves orientation on $X$ and so that $\Ga_i=G_i\cap\Ga$ preserves orientation on $X_i$ and $X$. We've already explained why this is possible. Here we show that we can find finite-index $\Ga'<\Ga$ so that if $\ga'\in\Ga'$ and $\ga'X_1\cap X_2\neq\vn$, then $\ga'\in G_2^+(\R)G_1^+(\R)$. This will prove that $\Ga'$ satisfies ($\ddagger$) and finish part (a) of our proof of Theorem \ref{thm:cycle}. 

{\bf Orientations and spinor norm.} Recall that $G_i^+(\R)<G_i(\R)$ denotes the subgroup that preserves orientation on $X_i$ and on $X$. We explain how to determine these groups in our situation. For this, the spinor norm plays an important role. 

Let $F/\Q$ be a field extension (we will only use $F=\Q,\R,\Q_p$). Then $G(F)$ is a group of orthogonal transformations of the quadratic space $V_F=\Lambda\ot_\Z F$, and the spinor norm $\theta_F:G(F)\ra F^\ti/(F^\ti)^2$ is a homomorphism, defined as follows. Any $g\in G(F)$ can be expressed as a product of reflections $g=R^{x_1}\cdots R^{x_k}$, where $R^{x}$ denotes the reflection about the orthogonal complement of $x\in V_F$. Then one defines
\[\theta_F(g)=\prod_{i=1}^k x_i\cdot x_i\text{ mod} (F^\ti)^2,\]
which is well-defined independent of the choice of reflections. For more information, see \cite[\S55]{omeara} and also \cite[\S4]{millson-raghunathan}.

In particular, $G(\R)\cong\SO(p,q)$ has two components, detected by the spinor norm $\ta: G(\R)\ra\R^\ti/(\R^\ti)^2$.

\begin{lem}\label{lem:spinor}
Let $G=\SO(B)$ with $G(\R)\cong\SO(p,q)$. If $p+q$ is even, then $G(\R)$ preserves the orientation on $X=G(\R)/K$. If $p+q$ is odd, then the orientation preserving subgroup of $G(\R)$ is the kernel of the spinor norm homomorphism. 
\end{lem}

From the lemma, it follows that $G_2(\R)\cong\SO(p,q-1)$ preserves orientation on $X_1$ if $p+q$ is odd and preserves orientation on $X$ if $p+q$ is even. Hence $g\in G_2^+(\R)$ if and only if $\theta(g)=1$. Lemma \ref{lem:spinor} is easy to check; its proof is similar to the proof of the following lemma, whose proof we give. 

\begin{lem}\label{lem:orientation}
The group $G_1(\R)$ preserves orientation on $X_1$. If $p+q$ is even, then $G_1(\R)$ preserves orientation on $X_1$. If $p+q$ is odd, then $g\in G_1(\R)$ preserves orientation on $X$ if and only if $\theta(g)=1$. 
\end{lem}

\begin{proof}
Whether or not $g\in G_1(\R)$ preserves orientation on $X_1$ or $X$ depends only on the component of $g\in G_1(\R)\cong(\R^\ti)^p$. Thus it suffices to consider the action of elements of $G_1(\R)\cap K\cong\{\pm1\}^p$. This allows one to reduce to the tangent space at the basepoint $o\in X$, where the action of $K$ is the adjoint action. 

The Lie algebra of $G(\R)$ decomposes $\mf g=\mf k\oplus\mf p$, where $\mf k$ is the Lie algebra of $K$ and $\mf p\cong T_oX$. We identify $G(\R)$ with the group of isometries of $\R^{p+q}$ with respect to the form whose matrix in the basis $(e_1,\ldots,e_p,f_1,\ldots,f_q)$ is 
\[\left(\begin{array}{cc}\id_{p}&0\\0&-\id_{q}\end{array}\right).\] We choose $K\cong S(\oO(p)\ti\oO(q))$ to be the obvious block diagonal subgroup. We can identify $\mf p$ with $p\ti q$ matrices $M_{p,q}$. The adjoint action of $\left(\begin{smallmatrix}k_1\\&k_2\end{smallmatrix}\right)\in K=S(\oO(p)\ti \oO(q))$ on $\mf p$ is given by $A\mapsto k_1Ak_2^{-1}$. 

Up to conjugation in $G(\R)$, we can choose $G_1$ so that $G_1(\R)=\prod_{i=1}^p \SO(\R\{e_i,f_i\})$. Then $G_1(\R)\cap K\cong\{\pm1\}^p$ is generated by maps $\de_k$, where $\de_k$ acts by $-\id$ on $\R\{e_k,f_k\}$ and by $\id$ on $\R\{e_k,f_k\}^\perp$. The subspace $T_oX_1\sbs T_oX$ is identified with the diagonal matrices in $M_{p,q}\cong T_oX$ (i.e.\ matrices with $a_{ij}=0$ for $i\neq j$). 

Now one computes: the adjoint action of $\de_k$ on $M_{p,q}$ is given by 
\[(a_{ij})\mapsto (b_{ij}), \>\>\>\text{ where }\>\>\>b_{ij}=\begin{cases}-a_{ij}&i=k\text{ or }j=k\text{ but not both}\\a_{ij}&\text{else.}
\end{cases}\] Thus the determinant of $\de_k$ acting on $M_{p,q}\cong T_oX$ is $(-1)^{p+q-2}$, and the action on $T_oX_1\cong$(diagonal matrices) is by the identity. This lemma follows directly from this computation. 
\end{proof}

As a consequence of Lemmas \ref{lem:spinor} and \ref{lem:orientation} we immediately obtain: 
\begin{cor}
Fix $\ga\in\Ga$ with $\ga X_1\cap X_2\neq\vn$. If $\ga=g_2g_1\in G_2(\R)G_1(\R)$ with $\theta(g_2)=\theta(g_1)=1$, then the intersection $\ga X_1\cap X_2$ is positive. 
\end{cor}

{\bf Eliminating intersections.} The final step in the proof of part (a) of Theorem \ref{thm:cycle} is the following proposition. 

\begin{prop}\label{prop:eliminate}
Fix $\Ga$ as above. There exists a finite-index normal subgroup $\Ga'\lhd\Ga$ so that $\Ga'\cap\Ga_2\ga\Ga_1=\vn$ for every $\ga\in I(\Ga)\bs I^+(\Ga)$. 
\end{prop}

By Proposition \ref{prop:eliminate}, after replacing $\Ga$ by $\Ga'$, every intersection $\ga X_1\cap X_2$ has a positive sign by Lemma \ref{lem:endgame}. 

\begin{proof}[Proof of Proposition \ref{prop:eliminate}]
For a prime $p$, consider $G_i(\Z_p)$, and define $G_i^0(\Z_p)\sbs G_i(\Z_p)$ the subgroup on which the $p$-adic spinor norm is trivial. By \cite[\S4, Cor.\ 1]{millson-raghunathan} after replacing $\Ga$ by a finite-index subgroup, we can assume that $\theta$ is trivial on $\Ga$ (and hence also on $\Ga_1,\Ga_2<\Ga$). In particular, $\Ga_2\Ga_1\sbs G_2^0(\Z_p)G_1^0(\Z_p)$ for each $p$. 

Recall (c.f.\ Lemma \ref{prop:rationalproduct}) that for each $\ga\in I(\Ga)$, we can write $\ga=g_2g_1$ with $g_i\in G_i(\Q)$. This expression is unique because $G_1,G_2$ are generic. Since $\theta(\ga)=1$, we have $\theta(g_2)=\theta(g_1)\in\R^\ti/(\R^\ti)^2$. Consider the subset $\{\ga_1,\ldots,\ga_m\}\sbs I(\Ga)$ of those $\ga_j$ for which $\ga_j=g_{2,j}g_{1,j}$ with $\theta(g_{2,j})\neq1$. Note that $\{\ga_1,\ldots,\ga_m\}$ contains the complement of $I^+(\Ga)$. To prove the proposition, we will find $\Ga'\lhd\Ga$ so that $\Ga'\cap\Ga_2\ga_j\Ga_1=\vn$ for each $j$. 

For each $j$, there exists a prime $p_j$ so that $\theta_{p_j}(g_{2,j})\neq1$, where $\theta_{p_j}:G(\Q_{p_j})\ra \Q_{p_j}^\ti/(\Q_{p_j}^\ti)^2$ is the $p$-adic spinor norm (i.e.\ if $x\in\Q$ is not a square, then there exists a prime $p$ so that $x$ is not a square in $\Q_p$). For each $j$, there exists $n_j$ so that 
\[\Ga(p_j^{n_j})\ga_j\cap G_2^0(\Z_{p_j})G_1^0(\Z_{p_j})=\vn.\]
To see this, note that the groups $G_i^0(\Z_{p_j})$ are compact in $G(\Q_p)$ and hence so too is their product. Thus since $\theta_{p_j}(g_{2,j})\neq1$, the element $\ga_j$ is not contained in $G_2^0(\Z_{p_j})G_1^0(\Z_{p_j})$ (again using that $G_1,G_2$ is generic so the expression $\ga_j=g_{2,j}g_{1,j}$ is unique), and so there is a $p$-adic neighborhood $\Ga(p_j^{n_j})$ of $\ga_j$ that is disjoint from $G_2^0(\Z_{p_j})G_1^0(\Z_{p_j})$. A similar argument appears in \cite[\S8]{avramidi-phan}. 

Consider $\Ga'=\bigcap\Ga(p_j^{n_j})$. For each $j$, by construction $\Ga'\ga_j\cap G_2^0(\Z_{p_j})G_1^0(\Z_{p_j})=\vn$. This implies that $\Ga'\ga_j\cap\Ga_2\Ga_1=\vn$ since $\Ga_2\Ga_1\sbs G_2^0(\Z_{p_j})G_1^0(\Z_{p_j})$. By Lemma \ref{lem:normal} we conclude that $\Ga'\cap\Ga_2\ga_j\Ga_1=\vn$. This completes the proof. 
\end{proof}

\subsection{Arrangements of flats and proof of Theorem \ref{thm:cycle}(b).} \label{sec:arrange}

So far, we've shown that we can find $G_1$ and $\Ga<G(\Z)$ so that the associated cycle $[Y_1]\in H_p(Y)$ is nontrivial. Here we show that given $n\ge1$, we can find $G_1^{1},\ldots,G_1^{n}$ and $\Ga<G(\Z)$ so that the associated cycles $[Y_1^1],\ldots,[Y_1^{n}]\in H_p(Y)$ are linearly independent. We will assume $2\le p\le q$. The case $p=1$ (i.e.\ $X$ is hyperbolic space) is easy. 

The argument will mostly take place in the symmetric space $X$. For this reason we change our notation slightly, denoting maximal flats (previously $X_1$) by $F\sbs X$ and ``hyperplanes" (previously $X_2$) by $H\sbs X$. (Calling $H$ an hyperplane is misleading since its codimension is $p$. However, $H$ is the group preserving a hyperplane $P\sbs V$, so in that sense the name is perhaps reasonable.)

Our approach is as follows. 
\begin{enumerate}
\item For each $n\ge1$, we find collections $\{F_i\}_{1}^n$ and $\{H_i\}_1^n$ of flats and hyperplanes in $X$ so that the intersection matrix $(F_i\cdot H_j)$ is invertible. The groups $G_{F_i}, G_{H_i}$ (the analogues of $G_1,G_2$ before) will be defined over $\R$, but not necessarily defined over $\Q$. 
\item We explain why we can perturb $\{F_i\}$ (resp.\ $\{H_i\}$) so that they descend to compact (resp.\ properly embedded) submanifolds $\bar F_i,\bar H_i$ of $Y=\Ga\bs X$ for some $\Ga$. The proof of part (a) of Theorem \ref{thm:cycle} will then allow us to replace $\Ga$ by a finite-index subgroup so that the intersection matrix $(\bar F_i\cdot\bar H_j)$ is invertible. From this we conclude that the cycles $[\bar F_1],\ldots,[\bar F_n]$ are linearly independent. 
\end{enumerate} 

Before we carry out this plan, we describe $X$ in terms of a Grassmannian, and explain when a hyperplane and a flat intersect transversely. 

For much of this section, the integral structure $\Lam\sbs V$ will not play a role, so we will identify $V\cong\R^{p,q}$ with standard basis of orthogonal vectors $\R^{p,q}=\lan e_1,\ld,e_p,f_1,\ld,f_q\ran$ with $e_i\cdot e_i=1$ and $f_j\cdot f_j=-1$.

{\bf Flats, hyperplanes, and the Grassmannian of positive $p$-planes.} Define $\Gr_p(V)$ to be the space of $p$-dimensional subspaces of $V$ on which the form is positive definite, topologized as a subspace of the Grassmannian. The Lie group $\SO(V)$ acts transitively on $\Gr_p(V)$ with stabilizer a maximal compact, so the symmetric space $X=\SO(V)/K$ is isomorphic to $\Gr_p(V)$. 

Given a decomposition $V=P\oplus L$, where $L$ is a negative line, we define a ``hyperplane"
\[H=\{W\in\Gr_p(V): W\sbs P\}.\]
Given a decomposition $V=U_1\oplus\cdots\oplus U_p\oplus N$, where $U_i\cong\R^{1,1}$ and $N$ is negative-definite, we define a flat 
\[F=\{W\in\Gr_p(V): W=\oplus_{i=1}^pW\cap U_i\}\]
As a sanity check, one can see that $F\cong\R^p$ as follows. If $W\in F$, then $W\cap U_i$ is a positive line for each $i$. The space of positive lines in $\R^{1,1}$ is homeomorphic to $\R$. As one varies the choice of $W\cap U_i$ for each $i$, one gets a subspace of $\Gr_p(V)$ homeomorphic to $\R^p$. 

The following lemma characterizes when $H$ and $F$ intersect and when that intersection is transverse. Its proof is easy. 
\begin{lem}\label{lem:flat-hyperplane}
Let $V=P\oplus L$ and $V=U_1\oplus\cd\oplus U_p\oplus N$ be two decompositions as above, and let $H$ and $F$ be the associated hyperplane and flat. 
\begin{enumerate}
\item[(i)] If $P\cap U_i$ does not contain a positive line for some $i$, then $H\cap F=\vn$. 
\item[(ii)] If $P\cap U_i$ is equal to a positive line for each $i$, then $H$ and $F$ intersect transversely in a single point. 
\item[(iii)] If $P\cap U_i$ contains a positive line for every $i$, and $P\cap U_j=U_j$ for some $j$, then $\dim H\cap F\ge1$. 
\end{enumerate} 
\end{lem}

In the setup of Proposition \ref{prop:generic}, the condition $\la\notin U^\perp$ and the assumption that $\tau$ acts irreducibly on $U$ implies that when one considers the $\tau$-invariant decomposition $U\ot\R=U_1\oplus\cdots \oplus U_p$, then the projection of $\la$ to each $U_i$ is nonzero; thus $P\cap U_i$ is a proper subspace of $U_i$ for each $i$. Consequently, the condition $\la\notin U^\perp$ in Proposition \ref{prop:generic} corresponds to cases (i) and (ii) in Lemma \ref{lem:flat-hyperplane}. 

{\bf A good arrangement.} For each $n\ge1$, we construct a sequence of hyperplanes $\{H_\ell\}_{\ell=1}^n$ and flats $\{F_k\}_{k=1}^n$ defined over $\R$ so that the intersection matrix $(H_\ell\cdot F_k)$ is invertible. 

To begin, let $F_0$ be the flat corresponding to 
\[\R^{p,q}=\lan e_1,f_1\ran\op\cd\op \lan e_p,f_p\ran\op\lan f_{p+1},\ld,f_q\ran=:U_1\oplus\cdots\oplus U_p\oplus N.\]
Next we define a hyperplane $H_0$. First let $\phi:\R^{p,q}\ra\R^{p,q}$ be an automorphism that acts by the identity on $\lan e_1,f_1\ran^\perp$, and whose restriction to $\lan e_1,f_1\ran$ expands $\lan e_1+f_1\ran$ and contracts $\lan e_1-f_1\ran$. For each $m\ge0$, define $a_m,b_m$ by $\phi^m(e_1)=a_m\>e_1+b_m\>f_1$. Then $\phi^m(f_1)=b_me_1+a_mf_1$. We will also use the shorthand $e_1^m:=\phi^m(e_1)$ and $f_1^m:=\phi^m(f_1)$. By definition, $a_m^2-b_m^2=1$ for each $m$ (hence $a_m>b_m$), and $a_m,b_m\ra\infty$ and $\fr{a_m}{b_m}\ra 1$ as $m\ra\infty$. Fix $m\gg0$ (to be chosen later, depending on $n$). Let $H_0$ be the hyperplane defined by the decomposition $\R^{p,q}=P\op L$, where 
\[L=\lan f_1^m+f_2+\cd+f_p\ran,\]
and 
\[P=L^\perp=\lan e_1^m,e_2,\ld,e_p,f_1^m-f_2,f_2-f_3,\ld,f_{p-1}-f_p,f_{p+1},\ldots,f_q\ran.\]

Define flats $F_k$ for $k\ge1$ by rotating $F_0$ as follows. Fix $-1\ll\ta<0$ (to be chosen later, depending on $n$). Let $r:\R^{p,q}\ra\R^{p,q}$ be the rotation that is the identity on $\lan e_1,e_2,f_1,f_2\ran^\perp$ and restricts to each of $\lan e_1,e_2\ran$ and $\lan f_1,f_2\ran$ as a counter-clockwise rotation of angle $\ta$. Note $r\in\SO(p)\ti\SO(q)$ (note also that to define $r$ we have used $p,q\ge2$). For each $k\ge1$, define $F_k=r^k(F_0)$.

\begin{lem}[Intersection pattern]\label{lem:intersect} The intersection $H_0\cap F_0$ is nonempty. For $k\ge1$, if 
\begin{equation}\label{eqn:intersect}-(a_m+b_m)\le\tan(k\ta)\le-(a_m-b_m), \end{equation}
then $H_0\cap F_k=\vn$. 
\end{lem}

Before we prove Lemma \ref{lem:intersect}, we show that it allows us to find a desired arrangement of $H_\ell, F_k$. 

Observe that for each $n\ge1$, we can choose $m\gg0$ and $-\fr{\pi}{4}\ll\ta<0$ so that (\ref{eqn:intersect}) is true for $k=1,\ld,n$. Thus for $k=0,\ld,n$, we have $H_0\cap F_k\neq\vn$ if and only if $k=0$. Now define $H_\ell=r^\ell(H_0)$. If $k\ge\ell$, then 
\[H_\ell\cap F_k\neq\vn\>\>\Lra\>\>r^\ell(H_0)\cap r^k(F_0)\neq\vn\>\>\Lra\>\>H_0\cap r^{k-\ell}(F_0) \neq\vn\>\>\Lra\>\>k=\ell.\]Consequently, the intersection matrix $(H_\ell\cdot F_k)$ is lower triangular with 1's on the diagonal. This matrix is invertible, as desired. 

\begin{proof}[Proof of Lemma \ref{lem:intersect}]
The first statement is easy: the intersection of $H_0$ and $F_0$ is the $p$-plane $W=\lan e_1^m,\ldots,e_p\ran$. 

Now we prove the second statement. The flat $F_k$ corresponds to the decomposition 
\[\R^{p,q}=U_1^k\op U_2^k\op U_3\op\cd\op U_p\op N,\]
where $U_i^k=r^k(U_i)\sbs\R^{p,q}$  for $i=1,2$. Note for $i=1,2$ that $U_i^k$ is spanned by $r^k(e_i), r^k(f_i)$, and $r^k(e_1)=\cos(k\ta)e_1+\sin(k\ta)e_2$ and $r^k(e_2)=-\sin(k\ta)e_1+\cos(k\ta)e_2$, and the same formulas hold when $e_1,e_2$ are replaced by $f_1,f_2$.  

We will compute $P\cap U_1^k$ and see under what conditions the intersection is a positive line. If $v\in P\cap U_1^k$, then we can write 
\begin{equation}\label{eqn:Pcoeff}\begin{array}{llll}v&=&A_1(a_me_1+b_mf_1)+A_2e_2+B_1(b_me_1+a_mf_1-f_2)\\[2mm]&&+A_3e_3+\cd+A_pe_p+B_2(f_2-f_3)+\cd+B_{p-1}(f_{p-1}-f_p)+B_{p+1}f_{p+1}+\cd+B_qf_q\end{array}\end{equation}
and also 
\begin{equation}\label{eqn:U1k}v=X\big(\cos(k\ta)e_1+\sin(k\ta)e_2\big)+Y\big(\cos(k\ta)f_1+\sin(k\ta)f_2\big).\end{equation}
Since the coefficients on $e_3,\ld,e_p$ and $f_3,\ld,f_q$ are zero in (\ref{eqn:U1k}), $A_i=0\text{ for }i\ge3$  and $B_j=0$ for $j\ge2$. Then setting equations (\ref{eqn:Pcoeff}) and (\ref{eqn:U1k}) equal (and changing notation on the coefficients slightly), 
\[(Aa_m+Bb_m)e_1+(Ab_m+Ba_m)f_1+Ce_2-Bf_2=X\cos(k\ta)e_1+Y\cos(k\ta)f_1+X\sin(k\ta)e_2+Y\sin(k\ta)f_2.\]
We can simplify the corresponding system of equations to 
\[X\cos(k\ta)b_m+Y\sin(k\ta)b_m^2=Y\cos(k\ta)a_m+Y\sin(k\ta)a_m^2,\]
so that 
\[X=\left[\fr{a_m}{b_m}+\fr{1}{b_m}\tan(k\ta)\right]\>Y.\]
We want to know if $X^2-Y^2$ is positive or negative. Since 
$X^2-Y^2=\left(\left[\fr{a_m}{b_m}+\fr{\tan(k\ta)}{b_m}\right]^2-1\right)Y^2$, 
this is nonpositive if and only if 
$-1\le\fr{a_m}{b_m}+\fr{\tan(k\ta)}{b_m}\le 1$. 
This inequality is equivalent to (\ref{eqn:intersect}). If it holds, then $H_0\cap F_k=\varnothing$ by Lemma \ref{lem:flat-hyperplane}. This completes the proof. 
\end{proof}

{\bf Cocompact flats and rational hyperplanes.} Now we explain how any flat/hyperplane in $X$ can be perturbed to one that descends to a properly immersed submanifold of $Y=\Ga\bs X$. This will allow us to perturb the arrangement constructed above to an arrangement that descends to $Y$. 

{\it Rational hyperplanes.} We say a hyperplane $H\sbs X$ is \emph{rational} or \emph{defined over $\Q$} if the line $L$ in the corresponding decomposition $V=P\op L$ is defined over $\Q$ (equivalently, $L$ is spanned by an integral vector $\la\in\Lam$). In this case, the subgroup of $G_H$ that preserves the decomposition $P\op L$ is defined over $\Q$. Furthermore, since the $G(\Q)$ orbit of a negative rational line is dense in the space of all negative lines in $V$, any hyperplane $H\sbs X$ can be approximated by a rational hyperplane (one way to say this: for any neighborhood $\Om$ of $\pa H$ in the visual boundary $\pa X$, there exists a rational hyperplane $H'$ so that $\pa H'\sbs \Om$). 

{\it Rational flats.} We say a flat $F\sbs X$ is \emph{rational} if its stabilizer is defined over $\Q$. In this case, it descends to a properly embedded submanifold of $Y=X/\Ga$ by \cite[Thm.\ D]{schwermer-survey}, c.f.\ Theorem \ref{thm:schwermer}. The condition that $F$ is rational is not enough for the quotient in $Y$ to be compact. However, as discussed in \S\ref{sec:dagger}, by \cite{prasad-raghunathan} there exists $\tau\in G(\Z)$ whose centralizer $C_\tau(\R)$ is a Cartan subgroup and $C_\tau(\R)/(\Ga\cap C_\tau(\R))$ is compact. The element $\tau$ will preserve some decomposition $\R^{p,q}=U_{\tau,1}\op\cd\op U_{\tau,p}\op N_\tau$. The $G(\Q)$ orbit of $(U_{\tau,1},\ld,U_{\tau,p})$ in the space of all $p$-tuples $(U_1,\ld,U_p)$ of orthogonal subspaces $U_i\cong\R^{1,1}\hra\R^{p,q}$ is dense (because $G(\R)$ acts transitively on such tuples and $G(\Q)\sbs G(\R)$ is dense). Thus any flat $F\sbs X$ can be approximated by a rational flat $F'$ that is compact in the quotient $X/\Ga$. 

In summary, to prove part (b) of Theorem \ref{thm:cycle}, given $n\ge1$, we start with the arrangement $\{F_k\}_1^n$ and $\{H_\ell\}_1^n$ of flats and hyperplanes in $X$ with the lower-triangular intersection pattern. Let $\bar F_k$ and $\bar H_\ell$ be the images of these submanifolds in $Y=X/\Ga$. First we perturb to get a new arrangement of rational flats and hyperplanes with the same intersection pattern so that each $\bar F_k$ is compact and each $\bar H_\ell$ is properly immersed in $Y$. By replacing $\Ga$ by a finite-index subgroup, we can ensure that $\bar F_k$ and $\bar H_\ell$ are oriented, embedded submanifolds (Theorem \ref{thm:schwermer}). Next we apply Proposition \ref{prop:rationalproduct} to each pair $(F_k,H_\ell)$ (and the corresponding subgroups $G_{F_k}, G_{H_\ell}<G$) to conclude that after replacing $\Ga$ by yet another finite-index subgroup, we can ensure that every $\ga\in I(\Ga)$ belongs to $G_{F_k}(\Q)G_{H_\ell}(\Q)$. Then by Remark \ref{rmk:disjoint}, $\bar H_\ell$ and $\bar F_k$ intersect if and only if $H_\ell$ and $F_k$ intersect, i.e.\ the intersection matrix $(\bar H_\ell\cdot\bar F_k)$ is also lower-triangular. Finally, we can pass to a further finite-index subgroup so that the diagonal entries in the intersection matrix are all positive by the argument of \S\ref{sec:ddagger}. Therefore, $(\bar H_\ell\cdot\bar F_k)$ is invertible, which implies that the homology classes $[\bar F_1],\ld,[\bar F_n]$ are linearly independent in $H_p(Y;\Q)$. This proves part (b) of Theorem \ref{thm:cycle}. 

\section{Vector bundles with arithmetic structure group }\label{sec:interpret}

By Corollary \ref{cor:main}, the classes produced in Theorem \ref{thm:cycle} give rise to characteristic classes of manifold bundles $W_g\ra E\ra B$ with fiber $W_g= \#_g(S^{2k}\ti S^{2k})$. In this section we explain what these characteristic classes measure. This gives a new perspective on the Millson--Raghunathan construction. This will play a role in \S\ref{sec:app}.

Before we begin, we recall the classification of lattices $\Lam\sbs\R^{p+q}$ with integral, unimodular, indefinite bilinear form; see e.g.\ \cite[Ch.\ II, \S4]{milnor-husemoller}. This classification is not strictly needed for what follows, but it is helpful to have these examples in mind. If the form is odd, then there exists a basis for $\Lam$, with respect to which the form has matrix $B_{p,q}$, where 
\begin{equation}\label{eqn:form}B_{p,q}=\left(\begin{array}{cc}I_p\\&-I_q\end{array}\right).\end{equation} If the form on $\Lam$ is even, then $q=p+8\ell$ for some $\ell\ge0$ and $\Lam$ is isomorphic to $H^{\op p}\op(-E_8)^{\op\ell}$, where 
\begin{equation}\label{eqn:form2}H=(\Z^2,\left(\begin{array}{cc}&1\\1&\end{array}\right))\end{equation} and $E_8$ is the unique positive-definite, even, unimodular lattice of rank 8. 

\subsection{Vector bundles with structure group $\SO(\Lam)<\SO(p,q)$.} 
Fix $1\le p\le q$ and set $n=p+q$. Fix a lattice $\Lam\cong\Z^n$ with an integral, unimodular bilinear form of signature $(p,q)$. Fix a primitive vector $\la\in\Lam$ such that $\lambda\cdot\lambda<0$. Set $V=\Lam\ot_\Z\R$. The goal of this section is to construct a characteristic class $c_{\la}\in H^p(B\Ga;\Q)$ for certain $\Ga<\SO(\Lam)$ and show that $c_\la$ is dual to a geometric cycle $[Y_2]$ as in \S\ref{sec:thm1}. 

Let $W\ra B$ be a oriented, real vector bundle with rank $n$. Let $W_b$ denote the fiber over $b\in B$. Assume that the structure group reduces from $\gl^+(V)$ to $\SO(V)$. This is equivalent to the existence of a fiberwise bilinear form $\bsym\be=\{\be_b\}_{b\in B}$ of signature $(p,q)$. We can always reduce the structure group from $\SO(V)$ to its maximal compact subgroup $K\cong S(O(p)\ti O(q))$ (because they are homotopy equivalent and so are their classifying spaces). Such a reduction defines a decomposition $W\cong U\oplus U^\perp$, where $U=\bigcup_{b\in B}U_b$ is a rank-$p$ subbundle and $\be_b:U_b\ti U_b\ra\R$ is positive definite for each $b$. Conversely, any positive rank-$p$ subbundle $U\sbs W$ defines a reduction of the structure group to $S(O(p)\ti O(q))$. The structure group of $W\ra B$ reduces to $\SO(\Lam)$ if and only if there exists a fiberwise lattice $\bsym\Lam=\bigcup_{b\in B}\Lam_b\sbs W$ where $\Lam_b$ (with its from $\be_b)$ is isometric to $\Lam$ for each $b\in B$. 

\begin{defn}
Fix $\Lam,V$ and $\la\in\Lam$ and a bundle $W\ra B$ with structure group $\SO(\Lam)$ as above. We say that a positive rank-$p$ subbundle $U\sbs W$ is \emph{orthogonal to $\la$ at $b\in B$} if there exists an isometry $\phi:\Lam\ra\Lam_b$ so that $U_b\sbs \phi(\la)^\perp$. If $U\sbs W$ is not orthogonal to $\la$ at any $b\in B$, then we say $U$ is \emph{nowhere orthogonal to $\la$}. 
\end{defn} 

The characteristic class we define will be an obstruction to finding $U\sbs W$ that is nowhere orthogonal to $\la$. We translate the problem of finding $U$ to a problem about finding a section of an associated bundle. 

Set $\pi=\pi_1(B)$. Let $\rho:\pi\ra\SO(\Lam)$ be the monodromy of $W\ra B$. The symmetric space $X=K\bs\SO(V)$ is homeomorphic to the Grassmannian 
\[\Gr_p(V)=\{V'\sbs V: V'\text{ is positive definite and} \dim V'=p\}\]
because $\SO(V)$ acts transitively on $\Gr_p(V)$ and the stabilizer of a point is isomorphic to a maximal compact subgroup $K\sbs\SO(V)$. 

The group $\pi$ acts on $X\cong \Gr_p(V)$ via the monodromy $\rho$. For a space $Z$ with a $\pi$-action, we denote the Borel construction $Z\sslash\pi:=\fr{\wtil B\ti Z}{\pi}$, where $\pi$ acts on the universal cover $\wtil B$ by deck transformations and $\pi$ acts on $Z$ by the given $\pi$-action, and the quotient is by the diagonal action. For any such $Z$, there is a fibration $Z\sslash\pi\ra *\sslash\pi= B$ with fiber $Z$. 

Observe that for $W\ra B$ with monodromy $\rho:\pi\ra\SO(\Lam)$, a section of the associated bundle $X\sslash\pi\ra B$ is equivalent to a positive rank-$p$ subbundle $U\sbs W$. 

Let $H_\la=\{V'\in\Gr_p(V): V'\sbs\la^\perp\}\sbs X$. This is the sub-symmetric space corresponding to the subgroup $\SO(\la^\perp)<\SO(V)$. The codimension of $H_\la$ in $X$ is $p$. 

By \cite[Thm.\ D]{schwermer-survey}, there exists a torsion-free, finite-index subgroup $\Ga_\la<\SO(\Lam)$ so that the $\Ga_\la$-orbit of $H$ is embedded and admits a $\Ga_\la$-invariant orientation. (The group $\Ga_\la$ is not uniquely defined by these properties, e.g.\ for every prime $\ell$, there exists $m>0$ so that the congruence subgroup $\ker\big[\SO(\Lam)\ra\SO(\Lam/\ell^m\Lam)\big]$ satisfies these properties. The construction below works for any choice of $\Ga_\la$.) 

Fix a finite index subgroup $\Ga<\Ga_\la$, and let $H_{\la,\Ga}$ be the $\Ga$-orbit of $H_\la$ in $X$. By replacing $B$ by a finite cover, we can ensure that $\rho(\pi)<\Ga$. Set $X_0=X\sm H_{\la,\Ga}$ and consider the bundle $X_0\sslash\pi\ra B$. If $W\ra B$ has a positive rank-$p$ subbundle $U\sbs W$ that is nowhere orthogonal to $\la$, then $X_0\sslash\pi\ra B$ has a continuous section. Now we can use obstruction theory to extract a characteristic class from this situation. For this, we need to know the first nontrivial homotopy group of $X_0$. 

\begin{lem}\label{lem:homotopy} Fix $k\ge0$. If $k\le p-2$, then $\pi_k(X_0)=0$. Furthermore, 
$\pi_{p-1}(X_0)\cong\bigoplus_{\pi_0(H_{\la,\Ga})}\Z$. 
\end{lem}
\begin{proof}
First assume $k\le p-2$. We show any map $S^k\ra X_0$ is homotopically trivial. Since $X\cong\R^{pq}$ is contractible, we obtain a diagram 
\begin{equation}\label{eqn:homotopy}
\begin{xy}
(0,10)*+{S^k}="A";
(15,10)*+{X_0}="B";
(0,0)*+{D^{k+1}}="C";
(15,0)*+{X}="D";
{\ar "A";"B"}?*!/_3mm/{i};
{\ar@{^{(}->} "A";"C"}?*!/_3mm/{};
{\ar "B";"D"}?*!/_3mm/{};
{\ar "C";"D"}?*!/_3mm/{j};
\end{xy}\end{equation}
Without loss of generality we may assume that $i$ and $j$ are smooth and $j$ is transverse to $H_{\la,\Ga}$. Since $k+1\le p-1$ and the codimension of $H_{\la,\Ga}$ is $p$, if $D$ is transverse to $H_{\la,\Ga}$, then $D\cap H_{\la,\Ga}=\vn$, which shows $i$ is homotopically trivial in $X_0$. 

By the Hurewicz theorem, $\pi_{p-1}(X_0)\cong H_{p-1}(X_0)$. Define a homomorphism $\phi:\pi_{p-1}(X_0)\cong H_{p-1}(X_0)\ra\bigoplus_{\pi_0(H_{\la,\Ga})}\Z$ as follows. Choose an orientation on each component of $H_{\la,\Ga}$. Given $i:S^{p-1}\ra X_0$, extend to $D^p\ra X$ transverse to $H_{\la,\Ga}$, and compute the algebraic intersection of $D^p$ with each component of $H_{\la,\Ga}$. 

The map $\phi$ is obviously surjective: for each component of $H_{\la,\Ga}$, one can choose a $(p-1)$-sphere in its link, and the image of these generate $\bigoplus_{\pi_0(H_{\la,\Ga})}\Z$. For injectivity, it is well-known that if $D,H$ are oriented submanifolds of an oriented manifold $X$ that intersect transversely in a finite collection of points and their algebraic intersection number is 0, then $D$ can be replaced by a homologous submanifold $D'$ with $\pa D=\pa D'$ so that $D'\cap H=\vn$. This shows that if $[S^{p-1}\ra X_0]$ is in the kernel of $\phi$, then $[S^{p-1}\ra X_0]=0$ in $H_{p-1}(X_0)$. 
\end{proof}

Applying obstruction theory (see e.g.\ \cite[Ch.\ 7]{davis-kirk}), we can try to build a section of $X_0\sslash\pi\ra B$. Assume that $B$ is a CW complex. We start by choosing a section over the 0-skeleton of $B$ and work our way up inductively defining a section on the $k$-skeleton for $k\le p-1$ using the fact that $\pi_{k-1}(X_0)=0$ for $k\le p-1$. Once we reach the $p$-skeleton we meet the first measurable obstruction, which takes the form of a cocycle $C_{\la,\Ga}(W)\in H^{p}\big(B;\pi_{p-1}(X_0)\big)$. If $C_{\la,\Ga}(W)\neq0$, then $X_0\sslash\pi\ra B$ has no continuous section, and so $W\ra B$ does not have a positive rank-$p$ subbundle $U$ that is nowhere orthogonal to $\la$. This is useful, but we are interested in a less-refined, $\Z$-valued obstruction. 

Since $H_{\la,\Ga}$ has a $\Ga$-invariant orientation, there is a preferred generator of each coordinate of $\bigoplus_{\pi_0(H_{\la,\Ga})}\Z$. We use this to define an augmentation map $\bigoplus_{\pi_0(H_{\la,\Ga})}\Z\ra\Z$. The augmentation map induces a map $H^{p}\big(B;\pi_{p-1}(X_0)\big)\ra H^{p}(B;\Z)$, which sends $C_{\la,\Ga}(W)$ to a class $c_{\la,\Ga}(W)\in H^p(B;\Z)$. 

\begin{prop}
Fix $\Lam,V$, $\la\in\Lam$, and $\Ga_\la<\SO(\Lam)$ as above. Let $B$ be a CW complex and let $W\ra B$ be a vector bundle with structure group $\Ga<\Ga_\la$. If $c_{\la,\Ga}(W)\neq0$ in $H^p(B;\Z)$, then $W\ra B$ has no positive, rank-$p$ subbundle $U\sbs W$ that is nowhere orthogonal to $\la$. Equivalently, for every positive, rank-$p$ subbundle $U\sbs W$, there exists $b\in B$ so that $U$ is orthogonal to $\la$ at $b$. 
\end{prop}

{\it Remark.} If $C_{\la,\Ga}=0$, then there exists of a section of $X_0\sslash\pi\ra B$, but this does not ensure that there exists $U\sbs W$ that is nowhere orthogonal to $\la$ since $X_0=X\setminus H_{\la,\Ga}$ is not the complement of the full orbit of $H_\la$ under $\SO(\Lam)$. Note also that if $\Ga'\sbs\Ga$, then $H_{\la,\Ga'}\sbs H_{\la,\Ga}$, so it is possible that $C_{\la,\Ga}\neq0$ but $C_{\la,\Ga'}=0$ (and similarly for $c_{\la,\Ga}$ and $c_{\la,\Ga'}$). 

If $B$ is a closed, oriented $p$-manifold, then we can evaluate $c_{\la,\Ga}(W)\in H^p(B;\Z)$ on the fundamental class to get an integer $\lan c_{\la,\Ga}(W),[B]\ran\in\Z$, which is computed as follows. We have a diagram 
\begin{equation}\label{eqn:reinterp}\begin{xy}
(-20,0)*+{X}="A";
(0,0)*+{X\sslash\pi}="B";
(20,0)*+{X/\Ga}="C";
(0,-10)*+{B}="D";
{\ar"A";"B"}?*!/_3mm/{};
{\ar "B";"C"}?*!/_3mm/{p};
{\ar "B";"D"}?*!/^3mm/{};
{\ar@{-->}@/_/ "D";"B"}?*!/^3mm/{u};
\end{xy}\end{equation}

Here $u$ is a section corresponding to a positive, rank-$p$ subbundle $U\sbs W$, and the map $p$ is the composition $X\sslash\pi=\fr{\wtil B\ti X}{\pi}\ra X/\rho(\pi)\ra X/\Ga$ (the first map collapses collapses $\wtil B$ to a point). Let $\bar H_{\la,\Ga}$ be the image of $H_\la$ in $X/\Ga$. By our choice of $\Ga_\la$ and the assumption $\Ga<\Ga_\la$, the inclusion $\bar H_{\la,\Ga}\hra X/\Ga$ is a proper embedding, c.f.\ \cite[Thm.\ D]{schwermer-survey}. Now tracing through the definitions, one finds that $\lan c_{\la,\Ga}(W),[B]\big\ran$ is equal to the algebraic intersection number of $p\circ u(B)$ with $\bar H_{\la,\Ga}$ in $X/\Ga$.

Applying the above construction to the universal bundle over $B\Ga$, we see that $c_{\la,\Ga}\in H^p(B\Ga)\cong H^p(X/\Ga)$ is dual to the cycle $\bar H_{\la,\Ga}$, which is a locally symmetric space for a nonuniform lattice in $\SO(\la^\perp)\cong\SO(p,q-1)$. In \S\ref{sec:thm1}, we showed that there exists $\Ga<\Ga_\la$ so that $[\bar H_{\la,\Ga}]\in H^{\text{cl}}_{pq-p}(X/\Ga)$ is nonzero. Then $c_{\la,\Ga}$ is also nontrivial.

\subsection{Vector bundles with structure group $\Sl_n(\Z)$.} The construction of the previous section can be repeated in other situations. Here we remark on a version for vector bundles with structure group $\Sl_n(\Z)$. We will use this in \S\ref{sec:app} to give an application similar to Corollary \ref{cor:main} to odd-dimensional manifolds. 

Fix the standard lattice $\Z^n<\R^n$. Let $\de=(P,L)$ denote be a pair of subspaces of $\R^n$ defined over $\Q$ such that $\R^n=P\op L$ and $\dim L=1$. For every such $\de$, we will associate a finite index subgroup $\Ga_\de<\Sl_n(\Z)$ and for every $\Ga<\Ga_\de$ we will define a characteristic class $c_{\de,\Ga}\in H^{n-1}(B\Ga;\Z)$ for real vector bundles $W\ra B$ with structure group in $\Ga$. 

Suppose $W\ra B$ is a real oriented vector bundle of rank $n$. The structure group reduces from $\gl_n^+(\R)$ to $\Sl_n(\Z)$ if and only if $W$ admits a fiberwise lattice $\bsym\Lam$. A reduction of the structure group from $\gl_n^+(\R)$ to its maximal compact $\SO(n)$ corresponds to a fiberwise inner product $\bsym\be$ on $W$. 

\begin{defn}Fix $\de=(P,L)$ and $W\ra B$ and $\bsym\Lam\sbs W$ as above. For a fiberwise inner product $\bsym\be$, we say that a $(P,L)$ is \emph{$\bsym\be$-orthogonal at $b\in B$} if there exists an isomorphism $\phi:(\R^n,\Z^n)\ra (W_b,\Lam_b)$ so that $\phi(P)$ and $\phi(L)$ are orthogonal with respect to $\be_b$. If $(P,L)$ is not $\bsym\be$-orthogonal at any $b\in B$, we say $(P,L)$ is \emph{nowhere $\bsym\be$-orthogonal}.
\end{defn}

We can translate the problem of finding an inner product $\bsym\be$ so that $(P,L)$ is nowhere $\bsym\be$-orthogonal to a problem of finding a section of an associated bundle. Let $X=\SO(n)\bs \Sl_n(\R)$. This symmetric space can be identified with the space of unit volume inner products on $\R^n$.  There is a bijective correspondence between fiberwise inner products $\bsym\be$ on $W\ra B$ and sections of $X\sslash\pi\ra B$, where $\pi=\pi_1(B)$ acts on $X$ via the monodromy $\rho:\pi\ra\Sl_n(\Z)$.  

Consider the submanifold $H_\de=\{\text{inner products such that } \R^n=P\op L\text{ is orthogonal}\}\sbs X,$ 
which is a sub-symmetric space for $\Sl_{n-1}(\R)\ti\R$. By \cite[Thm.\ D]{schwermer-survey}, we can find a torsion-free subgroup $\Ga_\de<\Sl_n(\Z)$ so that the $\Ga_\de$-orbit of $H_\de$ in $X$ is embedded and has a $\Ga_\de$-invariant orientation. Fix a finite-index subgroup $\Ga<\Ga_\de$. Denote the $\Ga$ orbit of $H_\de$ in $X$ by $H_{\de,\Ga}$, and set $X_0=X\sm H_{\de,\Ga}$. We replace $B$ with a finite cover so that the monodromy $\rho:\pi\ra\Sl_n(\Z)$ factors through $\Ga$. 

If $W\ra B$ admits an inner product $\bsym\be$ so that $(P,L)$ is nowhere $\bsym\be$-orthogonal, then $X_0\sslash\pi\ra B$ admits a continuous section. Similar to Lemma \ref{lem:homotopy}, we compute $\pi_k(X_0)=0$ for $k\le n-3$ and $\pi_{n-2}(X_0)\cong\bop_{\pi_0(H_{\de,\Ga})}\Z$. Then there is an obstruction class $C_{\de,\Ga}(W)\in H^{n-1}(B;\pi_{n-2}(X_0))$, which maps to a class $c_{\de,\Ga}(W)\in H^{n-1}(B;\Z)$ under the map induced by the augmentation $\bop_{\pi_0(H_{\de,\Ga})}\Z\ra\Z$. 

We summarize the above discussion with the following proposition. 
\begin{prop}
Fix $\de=(P,L)$ and $\Ga_\de<\Sl_n(\Z)$ as above. Let $B$ be a CW complex and let $W\ra B$ be a vector bundle with structure group $\Ga<\Ga_\de$. If $c_{\de,\Ga}(W)\neq0$ in $H^{n-1}(B;\Z)$, then $W\ra B$ does not admit an inner product $\bsym\be$ so that $(P,L)$ is nowhere $\bsym\be$-orthogonal. Equivalently, for every inner product $\bsym\be$ on $W$ there exists $b\in B$ so that $(P,L)$ is $\bsym\be$-orthogonal at $b$. 
\end{prop}

The class $c_{\de,\Ga}\in H^{n-1}(B\Ga)\cong H^{n-1}(X/\Ga)\cong H_{(n^2-n)/2}^{\text{cl}}(X/\Ga)$ is dual to the cycle $[\bar H_{\de,\Ga}]\in H_{(n^2-n)/2}^{\text{cl}}(X/\Ga)$ represented by the image of $H_\de$ in $X/\Ga$. Compare with the discussion following (\ref{eqn:reinterp}). By a theorem of Avramidi--Nguyen-Phan \cite{avramidi-phan} for a subgroup $\Ga<\Ga_\de$ of sufficiently large index, the homology class $[\bar H_{\de,\Ga}]\in H_{(n^2-n)/2}^{\text{cl}}(X/\Ga;\Q)$ is nontrivial.

\section{Applications to manifold bundles}\label{sec:app}

In this section and the appendix, we give applications of Theorem \ref{thm:cycle}.

\subsection{4-manifolds, K3 surfaces bundles, and the global Torelli theorem} \label{sec:K3}

Let $M$ be a closed oriented 4-manifold. As in the introduction, we use $\Lam_M$ to denote $H_2(M;\Z)/\text{torsion}$ with its intersection form. Assume that $\Lam_M$ is indefinite, and let $(p,q)$ be the signature. Up to switching the orientation, we may assume $p\le q$. By Theorem \ref{thm:cycle}, when $p$ is odd, there exists a finite-index subgroup $\Ga<\SO(\Lam_M)$ so that $H_p(\Ga;\Q)\neq0$.

\begin{qu}\label{q:dim4}
Does the image of $H_p\big(B\Diff^\Ga(M);\Q\big)\ra H_p\big(B\Ga;\Q\big)$ intersect the subspace spanned by flat cycles nontrivially? 
\end{qu}

One could ask a similar question for homeomorphisms or homotopy automorphisms. There does not seem to be a good reason for the answer to Question \ref{q:dim4} to be ``Yes", other than the evidence provided by Corollary \ref{cor:H2} and Theorem \ref{appendix:theorem} below. 

{\it Example.} Fix $1\le p\le q$ and let $M=M_{p,q}:=\big(\#_p\CP^2\big)\#\big(\#_q\ov{\CP^2}\big)$. Then the form on $\Lam_M$ has matrix $B_{p,q}$ (defined in (\ref{eqn:form})) and $\alpha: \Diff(M)\ra\oO(\Lam_M)$ is surjective if $p+q\le8$ or $p\ge2$. This follows from \cite[Thm.\ 2]{wall_diff4mfld}; which shows that $\al$ is surjective when $p+q\le 8$ or $M=N\#(S^2\ti S^2)$ is simply connected and $Q_N$ is indefinite. Since $\CP^2\#\ov{\CP^2}\#\ov{\CP^2}\cong(S^2\ti S^2)\#\ov{\CP^2}$ (see e.g.\ \cite[pgs.\ 124,151]{scorpan}), the hypotheses of Wall's theorem are true for $M_{p,q}$ when $p\ge2$. This gives many concrete examples to study Question \ref{q:dim4}. 

One does not necessarily need to restrict to flat cycles. In particular, for $M_{1,q}$, the group $\oO(\Lam_M)$ is a nonuniform lattice in $\oO(1,q)$. One easy source of homology of finite-index subgroups $\Ga\sbs\oO(\Lam_M)$ are classes $[T]\in H_{q-1}(\bb H^q/\Ga;\Q)\cong H_{q-1}(\Ga;\Q)$ represented by a cusp cross-section $T\sbs\bb H^q/\Ga$. When $\bb H^q/\Ga$ has at least 2 cusps, these classes are always nontrivial. Note that $T$ is finitely covered by a flat torus, but in contrast to flat cycles $T\hra\bb H^q/\Ga$ is not an isometric embedding. Even so, it would be interesting to understand the analogue of Question \ref{q:dim4} for these homology classes. Note that this question is most reasonable for $q<10$ since the image of $\al$ is infinite index in $\oO(\Lam_M)$ when $q\ge10$ \cite{friedman-morgan}. 

In the remainder of this section we study Question \ref{q:dim4} in the case $M$ is a K3 surface (i.e.\ a smooth 4-manifold diffeomorphic to a K3 surface). Here the form on $\Lam_M$ has matrix $H^{\op3}\oplus(-E_8)^{\op2}$ (the notation is explained in (\ref{eqn:form2})). Then $\SO(\Lam_M)$ is a lattice in $\SO(3,19)$. We will be interested in the maps 
\begin{equation}\label{eqn:K3}H_*(B\Diff^\Ga(M);\Q)\xra{\al_1} H_*(B\pi_0\Diff^\Ga(M);\Q)\xra{\al_2}H_*(B\Ga;\Q),\end{equation}
for $\Ga<\SO(\Lam_M)$. By Theorem \ref{thm:cycle}, we can find $\Ga$ and $z\neq0\in H_3(B\Ga;\Q)$. We will study whether or not $z$ is in the image of $\al_2$ and $\al_1\circ\al_2$.

\begin{thm}\label{thm:K3}
Let $M$ be a smooth oriented 4-manifold diffeomorphic to a K3 surface. There exists a finite-index subgroup $\Ga_M'<\SO(\Lam_M)$ so that for each finite-index subgroup $\Ga<\Ga_M'$ and for each $i\ge0$, the map $\al_2:H_i\big(B\pi_0\Diff^\Ga(M);\Q\big)\ra H_i\big(B\Ga;\Q\big)$ is surjective. 
\end{thm} 

Consequently, each flat cycle $z\neq0\in H_3(B\Ga;\Q)$ is in the image of $\al_2$. Theorem \ref{thm:K3} is a corollary of the global Torelli theorem (\cite{looijenga} and \cite[\S12.K]{besse}) and can be deduced from the discussion in \cite{giansiracusa}. 

\begin{proof}[Proof of Theorem \ref{thm:K3}]
Let $\Ga_M$ be the image of $\Diff(M)\ra\oO(\Lam_M)$. It is known \cite{matumoto_k3} that $\Ga_M$ is finite index in $\oO(\Lam_M)$. To prove the theorem, it suffices to show that the surjection $\Diff(M)\ra\Ga_M$ splits over a finite-index subgroup $\Ga_M'$. 

Let $\Ein(M)$ denote the space of unit-volume Einstein metrics on $M$, topologized as a subspace of all Riemannian metrics on $M$. One defines the \emph{homotopy moduli space} 
\[\ca M_{\Ein}(M):=\fr{\Ein(M)\ti E\Diff(M)}{\Diff(M)},\]
where $E\Diff(M)$ is the total space of the universal principal $\Diff(M)$ bundle over $B\Diff(M)$. There is a composition of maps
\begin{equation}\label{eqn:composition}\phi:\ca M_{\Ein}(M)\ra B\Diff(M)\ra B\Ga_M.\end{equation}

As explained in \cite[\S4-5]{giansiracusa}, the group $\pi_1\big(\ca M_{\Ein}(M)\big)$ is isomorphic to a finite-index subgroup $\Ga_M'<\Ga_M$, and $\phi$ induces the inclusion on $\pi_1$. 
\end{proof}

The last assertion in the proof will be further explained below (as part of the proof of Proposition \ref{prop:lift-K3}).

{\it Remark.} Since $\al_2$ is surjective, any homology of a lattice $\Ga<\SO(3,19)$ in the \emph{stable range} is also in the image of $\al_2$. Switching to cohomology, the stable cohomology can be described as the cohomology that is pulled back along the map
\[f:B\Ga\ra B\SO(3,19)\sim BS(O(3)\ti O(19))\ra B\SO(3).\]
Compare with \cite[\S3]{giansiracusa}. Recall $H^*(B\SO(3);\Q)\cong \Q[p_1]$, where $p_1\in H^4$ is the first Pontryagin class. According to the ranges in \cite{borel_cohoarith2}, $f$ induces an $H_i(-;\Q)$-isomorphism for $i< 1$. Unfortunately, this does not provide nontrivial elements of $H^*(B\Ga_M;\Q)$. (This is incorrectly quoted in \cite[Prop.\ 3.6]{giansiracusa}.) 

Theorem \ref{thm:K3} reduces Question \ref{q:dim4} to studying the image of $\al_1$. The author does not know of a single nontrivial class in the image of this map (or a single class that is not in the image of this map). In studying $\al_1$, we will focus on a particular type of flat cycle $z$. 

Set $V=\Lam_M\ot\R\cong H_2(M;\R)$, and let $X=\SO(V)/K$ be the symmetric space for $G=\SO(V)$. As discussed in \S\ref{sec:interpret}, there is a homeomorphism $X\cong \Gr_3(V)$. A vector $\de \in\Lam$ is called a \emph{root vector} if $\de\cdot\de=-2$. As in \S\ref{sec:interpret}, consider 
\[H_\de=\{V'\in\Gr_3(V): V'\sbs\de^\perp\}\sbs X.\]

Fix a root vector $\de$, and choose a rational flat $F\sbs X$ that intersects $H_\de$ transversely. (This can be done using the arguments of \S\ref{sec:arrange}.) By the construction of Theorem \ref{thm:cycle}, there exists $\Ga<\Ga_M'$ so that $F$ and $H_\de$ descend to homology cycles in $Y=\Ga\bs X$ that pair nontrivially. In particular, we have a nonzero class $z_0\in H_3(B\Ga;\Q)$. 

Now we discuss whether or not $z_0\in\im(\al_2\circ\al_1)$. One approach to this question is to consider the map (\ref{eqn:composition}) from the proof of Theorem \ref{thm:K3}. For each finite-index subgroup $\Ga<\Ga_M'$, define $\ca M_{\Ein}^\Ga(M)=\fr{T_{\Ein}^0(M)\ti E\Ga}{\Ga}$, where $\ca T_{\Ein}^0(M)$ is one of the two path components of $\ca T_{\Ein}^0(M)$ (these components are preserved by $\Ga_M'$). If $z_0$ is in the image of $\phi_*:H_3(\ca M_{\Ein}^\Ga(M))\ra H_3(B\Ga_M)$, then  $z_0\in\im(\al_2\circ\al_1)$. Unfortunately, the following proposition shows that this approach does not work. Nevertheless, we have an interesting Corollary \ref{cor:fiberwise-einstein}. 

\begin{prop}\label{prop:lift-K3}Let $M$ be a K3 surface. Fix $\Ga<\Ga_M'$ and $z_0\in H_3(B\Ga;\Q)$ as above. The class $z_0$ is not in the image of $\phi_*:H_3(\ca M_{\Ein}^\Ga(M);\Q)\ra H_3( B\Ga;\Q)$. 
\end{prop}

\begin{cor}\label{cor:fiberwise-einstein}
If $z_0\in H_3(B\Ga;\Q)$ is in the image of $H_3(B\Diff^\Ga(M);\Q)\ra H_3(B\Ga;\Q)$, then there exists a $K3$ bundle over a $3$-manifold that does not admit any fiberwise Einstein metric. 
\end{cor}

Of course it may be the case that $z_0$ is \emph{not} in the image of $H_3(B\Diff^\Ga(M);\Q)\ra H_3(B\Ga;\Q)$, in which case the corollary is vacuously true. In this situation, there is a different interesting corollary. 

\begin{cor}\label{cor:nielsen}
Let $M$ be a $K3$ surface. If there exists any flat cycle $z\in H_3(B\Ga;\Q)$ is not in the image of $H_3(B\Diff^\Ga(M);\Q)\ra H_3(B\Ga;\Q)$, then the surjection $\Diff(M)\ra\pi_0\Diff(M)$ is not split. 
\end{cor}

\begin{proof}[Proof of Corollary \ref{cor:nielsen}]
If a splitting exists, then $H_*(B\Diff^\Ga(M);\Q)\ra H_*(B\pi_0\Diff^\Ga(M);\Q)$ would be surjective for every $\Ga<\Ga_M$. Combining this with by Theorem \ref{thm:K3}, then $H_*(B\Diff^\Ga(M);\Q)\ra H_*(B\Ga;\Q)$ is also surjective for every $\Ga<\Ga_M$. This contradictions the assumption that some flat cycle is not in the image of $H_3(B\Diff^\Ga(M);\Q)\ra H_3(B\Ga;\Q)$. 
\end{proof}

\begin{proof}[Proof of Corollary \ref{cor:fiberwise-einstein}]
Suppose there is a class $w\in H_3(B\Diff^\Ga(M);\Q)$ whose image in $H_3(B\Ga;\Q)$ is $z_0$. Up to scaling, we can represent $w$ by a map of a manifold $h:B^3\ra B\Diff^\Ga(M)$. The pullback of the universal bundle by $h$ is a K3 bundle $M\ra E\ra B$. We claim it has no fiberwise Einstein metric. 

The homotopy moduli space $\ca M_{\Ein}(M)$ is a classifying space for K3 bundles with a fiberwise Einstein metric, so $E\ra B$ admits a fiberwise Einstein metric if and only if $h$ lifts to a map $\til h:B\ra\ca M^\Ga_{\Ein}(M)$. No such lift can exist by Proposition \ref{prop:lift-K3}. Thus $E\ra B$ has no fiberwise Einstein metric. 
\end{proof}

As remarked in Theorem \ref{thm:K3}, Giansiracusa \cite[\S4-5]{giansiracusa} proves that $\pi_1(\ca M_{\Ein}(M))\cong\Ga_M'$. We begin by explaining the proof of this fact, since it will be used to prove Proposition \ref{prop:lift-K3}. For this, we give a fuller description of the topology of $\ca M_{\Ein}(M)$, which is illuminated by the global Torelli theorem. For details see \cite{looijenga}, \cite[\S12.K]{besse}, and \cite[\S4-5]{giansiracusa}. In \cite[\S4.2]{giansiracusa}, Giansiracusa shows that
\begin{equation}\label{eqn:moduli}\ca M_{\Ein}(M)\cong \fr{\ca T_{\Ein}(M)\ti E\Ga_M}{\Ga_M},\end{equation}
where $\ca T_{\Ein}(M)$ is the \emph{Teichm\"uller space}. By definition $\ca T_{\Ein}(M)$ is the quotient $\Ein(M)/\Diff_1(M)$, where $\Diff_1(M)=\ker\big[\Diff(M)\ra\Ga_M\big]$. Here the action of $\Ga_M$ on $\ca T_{\Ein}(M)$ is induced from the action of $\Diff(M)$ on $\Ein(M)$ (by pulling back metrics). The global Torelli theorem determines $\pi_0(\ca T_{\Ein}(M))$ and the topology of each component: 
\begin{itemize} 
\item The space $\ca T_{\Ein}(M)$ has two homeomorphic components $\ca T_{\Ein}(M)\cong \ca T_{\Ein}^0(M)\sqcup\ca T_{\Ein}^0(M)$, and they are permuted by the action of $\Ga_M$. 
\item Let $\Ga_M'<\Ga_M$ be the index-2 subgroup that preserves the components of $\ca T_{\Ein}(M)$. There is a $\Ga_M'$-equivariant homeomorphism between $\ca T_{\Ein}^0(M)$ and a dense subspace of $X=\SO(V)/K$: 
\begin{equation}\label{eqn:torelli}\ca T_{\Ein}^0(M)\cong X\sm\bigcup_{\de\in\De}H_\de,\end{equation}
where $\De\sbs\Lam$ is the set of roots. 
\end{itemize} 

Using this description of $\ca T_{\Ein}(M)$, it follows that $\pi_1(\ca M_{\Ein}(M))\cong\Ga_M'$ by the long exact sequence in homotopy associated to (\ref{eqn:moduli}) together with the fact that the subspaces $H_\de\sbs X$ have codimension-3, so $\ca T_{\Ein}(M)$ is simply connected. 

\begin{proof}[Proof of Proposition \ref{prop:lift-K3}] We have fixed a particular torsion-free subgroup $\Ga<\Ga_M'$ and a flat cycle $z_0\in H_3(B\Ga;\Q)$, and we wish to show $z_0$ is not in the image of $H_3(\ca M_{\Ein}^\Ga(M);\Q)\ra H_3(B\Ga;\Q)$. Recall that $z_0$ has the special property that it pairs nontrivially with the image of a root hyperplane $H_\de$ in $Y=\Ga\bs X$.

Suppose for a contradiction that there exists $w\in H_3(\ca M_{\Ein}^{\Ga}(M);\Q)$ whose image in $H_3(B\Ga;\Q)$ is $z_0$. Since $\Ga$ is torsion free, $\ca M_{\Ein}^{\Ga}(M)\cong\fr{T_{\Ein}^0(M)\ti X}{\Ga}$. There is a diagram that commutes up to homotopy: 
\[\begin{xy}
(-20,0)*+{\ca M_{\Ein}^{\Ga}(M)}="A";
(10,0)*+{B\Diff^{\Ga}(M)}="B";
(40,0)*+{B\Ga}="C";
(-20,-15)*+{\Ga\bs\ca T_{\Ein}^0(M)}="D";
(10,-15)*+{\Ga\bs X}="E";
(40,-15)*+{B\Ga}="F";
{\ar"A";"B"}?*!/_3mm/{};
{\ar "B";"C"}?*!/_3mm/{};
{\ar@{^{(}->} "D";"E"}?*!/_4mm/{f_2};
{\ar "E";"F"}?*!/_3mm/{\sim};
{\ar "A";"D"}?*!/_3mm/{};
{\ar "A";"D"}?*!/^3mm/{f_1};
{\ar@{=} "F";"C"}?*!/^3mm/{};
\end{xy}\]
The map $f_1$ is a homotopy equivalence because $\Ga$ acts freely on $\ca T_{\Ein}(M)$ so $f_1$ is a fibration with contractible fiber $\cong X$. The map $f_2$ is the inclusion induced by (\ref{eqn:torelli}). The diagram commutes up to homotopy because the two compositions induce the same map on $\pi_1$ and $B\Ga$ is Eilenberg--Maclane space. 

Let $\bar H_\de$ be the image of $H_\de$ in $\Ga\bs X$, and let $\bar F$ be a totally geodesic submanifold representing the flat cycle $z_0$. By our choice of $H_\de$ and $F$, the algebraic intersection $\bar F\cdot\bar H_\de$ is nonzero. On the other hand, the existence of $w$ implies, by the diagram above, that there is a cycle $Z\ra\Ga\bs\ca T_{\Ein}^0(M)\sbs \Ga\bs X$ that is homologous to $\bar F$. Since the $Z\ra\Ga\bs X$ factors through $\Ga\bs\ca T_{\Ein}^0(M)$, the image of $Z$ is disjoint from $\bar H_\de$, which implies that $\bar F\cdot\bar H_\de=Z\cdot\bar H_\de=0$. This is a contradiction, so the class $w$ does not exist. 
\end{proof}

\begin{rmk}
In the Teichm\"uller space $\ca T_{\Ein}^0(M)\sbs X$, as one approaches one of the subsets $H_\de\sbs X$, topologically there is an embedded sphere $f:S^2\hra M$ with $f_*[S^2]=\de$ that is being collapsed to a point \cite{anderson}. The flat cycle $\bar F\sbs \Ga\bs X$ in the proof of Proposition \ref{prop:lift-K3} only lifts to $\Ga\bs\ca T_{\Ein}^0(M)$ after finitely many points are removed. From this, one obtains a K3-surface bundle over a 3-torus with finitely many punctures (note $\bar F$ is finitely covered by $\bb T^3$). One cannot extend the bundle over the punctures without introducing singularities. The natural object that exists over the torus with the punctures filled is a ``singular" K3 bundle, i.e.\ it is a fiber bundle away from finitely many points in the base, and at each point in this finite collection, the fiber is the space obtained from a K3 surface by collapsing some embedded 2-sphere (with self-intersection $-2$) a point. 
\end{rmk}

\subsection{2-dimensional cycles and the Mather--Thurston theorem.} 

To end this section, we mention another example/application of our ideas. 

\begin{cor}\label{cor:H2}Let $M$ be smooth manifold. Suppose that $\pi_0\Diff(M)$ is commensurable with $\Sl_3(\Z)$. Then for each $n\ge1$ there exists a finite-index subgroup $\Ga<\Sl_3(\Z)$ so that 
\[\dim H_2\big(B\Diff^\Ga(M);\Q\big)\ge n.\]
\end{cor}

In fact, as we will see in the proof, Corollary \ref{cor:H2} remains true if $B\Diff^\Ga(M)$ is replaced by the classifying space $B\Diff^\Ga(M)^\de$ of $\Diff^\Ga(M)$ with the discrete topology. This is a stronger conclusion than in Corollary \ref{cor:main}. 

For manifolds satisfying the hypothesis of Corollary \ref{cor:H2} on could consider $M=\#_3\big(S^k\ti S^{k+1}\big)$ for $k\ge4$, c.f.\ \cite[Thm.\ 13.3]{sullivan-infinitesimal}. 

Corollary \ref{cor:H2} follows by applying the following two theorems. The first theorem, due to Avramidi--Nguyen-Phan \cite{avramidi-phan}, is analogous to Theorem \ref{thm:cycle}. For a prime $p$ and $\ell\ge1$, denote the congruence subgroup $\ker\big[\Sl_d(\Z)\ra\Sl_d(\Z/p^\ell\Z)\big]$ by $\Ga_d(p^\ell)$. 

\begin{thm}[Avramidi--Nguyen-Phan]\label{thm:avramidi-phan}
Given a prime $p$ and an integer $n\ge1$, there exists $\ell_0>0$ so that if $\ell>\ell_0$, then $\dim H_{d-1}(\Ga_d(p^\ell);\Q)\ge n$. 
\end{thm}

Like in Theorem \ref{thm:cycle}, their homology comes from maximal flats in the associated symmetric space. We will focus on the case $d=3$, which is special because we can use the following theorem.

\begin{thm}[Mather, Thurston \cite{thurston-diff}]
Let $M$ be a smooth closed manifold. The group $\Diff_0(M)$ of diffeomorphisms isotopic to the identity is a simple group. 
\end{thm}

\begin{proof}[Proof of Corollary \ref{cor:H2}] 
Denote $\Mod(M):=\pi_0\Diff(M)$. The exact sequence 
\[1\ra\Diff_0(M)\ra\Diff(M)\ra\Mod(M)\ra1\] gives a 5-term exact sequence in group homology
\[H_2\big(\Diff(M)\big)\xra{\phi} H_2\big(\Mod(M)\big)\ra H_1\big(\Diff_0(M)\big)_{\Mod(M)}\ra H_1\big(\Diff(M)\big)\ra H_1\big(\Mod(M)\big)\ra0\]
By the Mather--Thurston theorem, $H_1\big(\Diff_0(M)\big)=\Diff_0(M)^{\text{ab}}=0$ so $\phi$ is surjective. Since the obvious composition $B\Diff(M)^\de\ra B\Diff(M)\ra B\Mod(M)$ is induced by the quotient $\Diff(M)\ra\Mod(M)$, the map $H_2\big(B\Diff(M)\big)\ra H_2 \big(B\Mod(M)\big)$ is also surjective. The same argument applies to $\Diff^\Ga(M)$. \end{proof}

One could go further and try to extend the computations of \cite{berglund-madsen-hs} to the manifolds $M=\#_g(S^k\ti S^{k+1})$ to show that the classes in Theorem \ref{thm:avramidi-phan} are in the image of $H_*(B\Diff^\Ga(M);\Q)\ra H_*(B\Ga;\Q)$ as in Theorem \ref{appendix:theorem}. For another example, one could look at the diffeomorphism groups of handlebodies, i.e.\ boundary-connected-sums of $D^{k+1}\ti S^k$ for $k\ge4$, c.f.\ \cite{botvinnik-perlmutter}.

\appendix

\renewcommand{\circled}[1]{\raisebox{.5pt}{\textcircled{\raisebox{-.4pt} {\footnotesize #1}}}}
\renewcommand{\id}{\ensuremath{\operatorname{id}}}
\renewcommand{\oO}{\ensuremath{\operatorname{O}}}
\newcommand{\oSO}{\ensuremath{\operatorname{SO}}}
\newcommand{\Der}{\operatorname{Der}}
\newcommand{\BGamma}{\mathrm{B\Gamma}}
\newcommand{\BG}{\mathrm{B}G}
\renewcommand{\ra}{\rightarrow}
\newcommand{\lra}{\longrightarrow}
\newcommand{\xlra}[1]{\overset{#1}{\longrightarrow}}
\newcommand{\interior}[1]{\ensuremath{\operatorname{int}(#1)}}
\newcommand{\bfL}{\ensuremath{\mathbf{L}}}
\newcommand{\bfQ}{\ensuremath{\mathbb{Q}}}
\newcommand{\bfZ}{\ensuremath{\mathbb{Z}}}
\newcommand{\bfC}{\ensuremath{\mathbb{C}}}
\newcommand{\bfR}{\ensuremath{\mathbb{R}}}
\newcommand{\oH}{\ensuremath{H}}
\newcommand{\oB}{\ensuremath{\mathrm{B}}}
\newcommand{\GL}{\operatorname{GL}}
\newcommand{\CE}{\mathrm{CE}}
\renewcommand{\Emb}{\ensuremath{\operatorname{Emb}}}
\renewcommand{\Diff}{\ensuremath{\operatorname{Diff}}}
\renewcommand{\BDiff}{\ensuremath{\operatorname{BDiff}}}
\newcommand{\BlockDiff}{\ensuremath{\operatorname{\widetilde{Diff}}}}
\newcommand{\BBlockDiff}{\ensuremath{\operatorname{B\widetilde{Diff}}}}
\newcommand{\BBlockDiffGamma}{\ensuremath{\operatorname{B\widetilde{Diff^\Gamma}}}}
\newcommand{\hAut}{\ensuremath{\operatorname{hAut}}}
\newcommand{\BhAut}{\ensuremath{\operatorname{BhAut}}}
\newcommand{\Diffo}{\ensuremath{\operatorname{Diff}^{\scaleobj{0.8}{+}}}}
\newcommand{\BDiffo}{\ensuremath{\operatorname{BDiff}^{\scaleobj{0.8}{+}}}}

\section{\for{toc}{Lifting cycles from arithmetic groups to diffeomorphism groups (by Manuel Krannich)}\except{toc}{Lifting cycles from arithmetic groups to diffeomorphism groups\\ \vspace{0.2cm} by Manuel Krannich}
} 

We denote by $\Diff(W_g)$ the topological group of orientation-preserving diffeomorphisms of the iterated connected sum $W_g=\sharp^g(S^n\times S^n)$ in the smooth topology. Fixing an embedded disc $D^{2n}\subset W_g$, we also consider the manifold $W_{g,1}=W_g\backslash \interior{D^{2n}}$ and its group of diffeomorphisms $\Diff_\partial (W_{g,1})$ fixing a neighborhood of the boundary pointwise, which is related to $\Diff(W_g)$ by a map $\Diff_\partial(W_{g,1})\ra\Diff(W_g)$ given by extending diffeomorphisms of $W_{g,1}\subset W_g$ via the identity. 
\subsection{The action on homology} The action of the group of diffeomorphisms $\Diff(W_g)$ on the middle homology $\oH_n(W_g)\cong\bfZ^{2g}$ preserves the unimodular hyperbolic $(-1)^n$-symmetric intersection form $\lambda\colon \oH_n(W_g)\otimes \oH_n(W_g)\ra\bfZ$ and thus gives rise to a map
\[\label{appendix:action}\pi_0\Diff(W_g)\lra\begin{cases}\Sp_{2g}(\bfZ)&n\text{ odd}\\\oO_{g,g}(\bfZ)&n\text{ even},\end{cases}\] whose image we denote by $G_g\le\GL_{2g}(\bfZ)$. If $n$ is even or $n=1,3,7$, this map is surjective and $G_g$ coincides with $\Sp_{2g}(\bfZ)$ or $\oO_{g,g}(\bfZ)$ depending on the parity of $n$, whereas the image $G_g\le\Sp_{2g}(\bfZ)$ for $n\neq1,3,7$ odd is the finite index subgroup $\Sp_{2g}^q(\bfZ)\le\Sp_{2g}(\bfZ)$ of matrices preserving the standard theta-characteristic (see e.g.\,\cite[Ex.\,5.5]{berglund-madsen-rational}). As the space of orientation-preserving embeddings $\Emb(D^{2n},W_g)$ is connected, the map $\pi_0\Diff_\partial(W_{g,1})\ra\pi_0\Diff(W_g)$ is surjective, so the images of these two groups in $\GL_{2g}(\bfZ)$ agree. Given a subgroup $\Gamma\le G_g$, we denote by $\Diff^{\Gamma}(W_g)\le \Diff(W_g)$ and $\Diff_\partial^{\Gamma}(W_{g,1})\le \Diff_\partial(W_{g,1})$ the preimages of $\Gamma$ with respect to the canonical maps to $G_g$.

The primary goal of this appendix is to prove the following.

\begin{thm}\label{appendix:theorem}Let $2n\ge6$ and $\Gamma\le G_g$ a finite index subgroup. The natural map
\[\oH^*(\BGamma;\bfQ)\lra\oH^*(\BDiff^{\Gamma}_\partial(W_{g,1});\bfQ)\] is injective in degrees
\[*\le 
\begin{cases}
2&\text{for }2n=6\\
n&\text{for }2n>6\text{ and }g< 4-c\\
n+\min(n-4,g-1+c)&\text{otherwise},
\end{cases}\]
where $c=0$ if is $n$ even and $c=1$ if is $n$ odd.
\end{thm}

\begin{nrem}Since the action of $\Diff^\Gamma_\partial(W_{g,1})$ on $\oH_n(W_g)$ factors through $\Diff^{\Gamma}(W_g)$, the same conclusion holds for $\BDiff^{\Gamma}(W_g)$ instead of $\BDiff^\Gamma_\partial(W_{g,1})$.
\end{nrem}

\subsection{Stable and unstable cohomology}The block-inclusion $\GL_{2g}(\bfZ)\subset\GL_{2g+2}(\bfZ)$ is covered by a map $\Diff_\partial(W_{g,1})\ra\Diff_\partial(W_{g+1,1})$ given by extending diffeomorphisms via the identity, so $\GL_{2g}(\bfZ)\subset\GL_{2g+2}(\bfZ)$ restricts to an inclusion $G_g\subset G_{g+1}$ and we obtain a map $\BDiff_{\partial}(W_{\infty,1})\ra\BG_\infty$ by taking (homotopy) colimits of $\BDiff_{\partial}(W_{g,1})\ra\BG_g$ in $g$. By work of Borel, Galatius--Randal-Williams, and Madsen--Weiss  \cite{borel_cohoarith,borel_cohoarith2, GRW-stable-moduli, MadsenWeiss}, the cohomology rings of $\BG_\infty$ and $\BDiff_{\partial}(W_{\infty,1})$ is a polynomial algebra concentrated in even degrees. Moreover, there are natural choices of polynomial generators for these rings with respect to which the induced map $\oH^*(\BG_\infty;\bfQ)\ra\oH^*(\BDiff_{\partial}(W_{\infty,1});\bfQ)$ corresponds to an inclusion of generators; this can for instance be seen by an index-theoretic argument (see e.g.\,\cite[Sect.\,2.4]{EbertRW}). Given a subgroup $\Gamma\le G_g$ of finite index, we have a commutative square
\begin{center}
\begin{tikzcd}
\oH^*(\BG_\infty;\bfQ)\arrow[r]\arrow[d]&\oH^*(\BGamma;\bfQ)\arrow[d]\\
\oH^*(\BDiff_{\partial}(W_{\infty,1});\bfQ)\arrow[r]&\oH^*(\BDiff^\Gamma_\partial(W_{g,1});\bfQ),
\end{tikzcd}
\end{center}
whose upper horizontal arrow is an isomorphism in a range of degrees growing with $g$ by a result of Borel \cite{borel_cohoarith,borel_cohoarith2}. In light of work of Harer and Galatius--Randal-Williams \cite{Harer,GRW-homological-stability-1}, the same holds for the lower horizontal morphism if $\Gamma=G_g$ and $2n\neq4$.\footnote{Recent work of Kupers--Randal-Williams \cite{KupersRandalWilliams} shows that, this holds for any finite index subgroup $\Gamma\le G_g$ as long as $2n\ge6$ and $\Gamma$ is not completely contained in the subgroup $\oSO_{g,g}(\bfZ)\subset\oO_{g,g}(\bfZ)$ if $n$ is even.} Moreover, the proof of \cref{appendix:theorem} will make apparent that the vertical arrows in the diagram are (compatibly) split injective in a range increasing with $n$ and any finite index subgroup $\Gamma\le G_g$ if $2n\ge 6$. As a result, the cokernel of the upper horizontal map---the so-called \emph{unstable cohomology} of $\BGamma$---injects in this range of degrees into the cokernel of the lower horizontal map and thus provides a source for unstable cohomology of $\BDiff^\Gamma_\partial(W_{g,1})$.

When varying $\Gamma\le G_g$ over finite index subgroups, the rational cohomology of $\BGamma$ in degree $g$ is arbitrarily large for $n$ even and $g$ odd; this is the main result of the body of this work (see \cref{thm:cycle}). For the full group $\Gamma=G_g$ on the other hand, little is known about the unstable cohomology, aside from some scattered classes: for instance, computations of Hain \cite{Hain3folds} and Hulek--Tomassi \cite{HulekTomassiVoronoi} show that for $n$ odd, there is a nontrivial unstable class in $\oH^6(\BG_3;\bfQ)$ and one in $\oH^{12}(\BG_4;\bfQ)$. By the above discussion, these classes remain nontrivial (and unstable) when pulled back to $\BDiff_\partial(W_{g,1})$ as long as $n$ is sufficiently large, so we obtain the following corollary.

\begin{cor}\label{appendix:corollary:unstableclasses}
For $n$ odd, the cokernel of the natural morphism
\[\oH^i(\BDiff_\partial(W_{\infty,1});\bfQ)\lra\oH^i(\BDiff_\partial(W_{g,1});\bfQ)\] is nontrivial for $(i,g)=(6,3)$ as long as $n>5$, and for $(i,g)=(12,4)$ if $n>8$.
\end{cor}

\begin{nrem}To the knowledge of the author, these are the first known unstable rational cohomology classes of $\BDiff_\partial(W_{g,1})$, aside from the case $2n=2$ of surfaces.
\end{nrem}

\subsection{The work of Berglund--Madsen}The proof of \cref{appendix:theorem} crucially relies on work of Berglund and Madsen \cite{berglund-madsen-rational}, who used a combination of classical surgery theory and rational homotopy theory to construct rational models for the classifying spaces $\BhAut^{\id}_\partial(W_{g,1})$ and $\operatorname{B\widetilde{Diff}}_{\partial,\circ}(W_{g,1})$ of the spaces of homotopy automorphisms and block diffeomorphisms homotopic to the identity.  Using these models, they proved that the rational cohomology rings of the classifying spaces of the full automorphism spaces $\hAut_\partial(W_{g,1})$ and $\BlockDiff_\partial(W_{g,1})$ are independent of $g$ in a range of degrees and studied the rational cohomology in this \emph{stable} range. As explained above, \cref{appendix:theorem} yields some information on $\oH^*(\BDiff_\partial(W_{g,1});\bfQ)$ in the \emph{unstable} range. Its proof involves relating $\oH^*(\BDiff_\partial(W_{g,1});\bfQ)$ to $\oH^*(\BBlockDiff_\partial(W_{g,1});\bfQ)$ by combining \cite{berglund-madsen-rational} with Morlet's lemma of disjunction as in \cite{UpperBound}, extending some arguments in \cite{berglund-madsen-rational} for spaces of automorphisms homotopic to the identity to the full automorphism spaces, in particular to show that the cohomology ring $\oH^*(\BhAut_\partial(W_{g,1});\bfQ)$ injects into $\oH^*(\BBlockDiff_\partial(W_{g,1});\bfQ)$ also in the unstable range (it is even a retract), and carrying out a spectral sequence argument involving low-degree computations of $\BhAut_\partial(W_{g,1})$ and some facts from the theory of arithmetic groups. 

\subsection{The proof of \cref{appendix:theorem}}We divide the proof of \cref{appendix:theorem} into three steps corresponding to three maps in a factorisation 
\begin{equation}\label{appendix:composition}
\BDiff^\Gamma_\partial (W_{g,1})\xlra{\circled{$1$}}\BBlockDiff^\Gamma_\partial (W_{g,1})\xlra{\circled{$2$}} \BhAut^{\cong,\Gamma}_\partial (W_{g,1})\xlra{\circled{$3$}} \BGamma,
\end{equation} 
which we explain in the following. Up to canonical equivalences, the topological group of block diffeomorphisms fixing a neighborhood of the boundary $\BlockDiff_\partial(W_{g,1})$ fits between $\Diff_\partial(W_{g,1})$ and the topological monoid of homotopy automorphisms fixing the boundary, so there are natural maps
$\Diff_\partial(W_{g,1})\ra\BlockDiff_\partial(W_{g,1})\ra\hAut_\partial(W_{g,1})$ (see e.g.\,\cite[Sect.\,4]{berglund-madsen-rational}), which explain the maps $\circled{$1$}$ and $\circled{$2$}$ for $\Gamma=G_g$, denoting by $\hAut^{\cong}_\partial (W_{g,1})\subset \hAut_\partial(W_{g,1})$ the union of components in the image of the map $\BlockDiff_\partial(W_{g,1})\ra\hAut_\partial(W_{g,1})$. By definition of the block-diffeomorphism group, the map $\Diff_\partial(W_{g,1})\ra\BlockDiff_\partial(W_{g,1})$ is surjective on path components (in fact, it is an isomorphism by Cerf's ``concordance implies isotopy"), so in particular the image of $\BlockDiff_\partial (W_{g,1})$ in $\GL_{2g}(\bfZ)$ coincides with the image $G_g$ of $\Diff_\partial(W_{g,1})$, which explains the map \circled{$3$} for $\Gamma=G_g$. For a general subgroup $\Gamma\le G_g$, this composition is defined analogously by restricting the components of the automorphism spaces involved to the preimage of $\Gamma$ of the canonical maps to $G_g$. 

In what follows, we examine the quality of the maps \circled{$1$}--\circled{$3$} in rational cohomology.

\subsection*{\circled{$1$}}Extending (block) diffeomorphisms of an embedded disc $D^{2n}\subset W_{g,1}$ to all of $W_{g,1}$ by the identity induces a commutative square
\begin{center}
\begin{tikzcd}
\BDiff_\partial(D^{2n})\arrow[r]\arrow[d]&\BDiff_\partial(W_{g,1})\arrow[d]\\
\BBlockDiff_\partial(D^{2n})\arrow[r]&\BBlockDiff_\partial(W_{g,1})
\end{tikzcd}
\end{center} whose induced map on vertical homotopy fibres
\[\BlockDiff_\partial(D^{2n})/\Diff_{\partial}(D^{2n})\lra\BlockDiff_\partial(W_{g,1})/\Diff_{\partial}(W_{g,1})\] is $(2n-4)$-connected by an application of Morlet's lemma of disjunction. For $g\gg0$, this can be combined with Berglund--Madsen's work \cite{berglund-madsen-rational} to conclude that $\BlockDiff_\partial(D^{2n})/\Diff_{\partial}(D^{2n})$ is $(2n-5)$-connected, as observed by Randal-Williams \cite[Sect.\,4]{UpperBound}. This in turn implies that $\BlockDiff_\partial(W_{g,1})/\Diff_{\partial}(W_{g,1})$ has no rational cohomology in degrees $*< 2n-4$ for \emph{all} $g\ge0$, which we combine with the homotopy pullback
\begin{center}
\begin{tikzcd}
\BDiff^\Gamma_\partial (W_{g,1})\rar\dar&\BDiff_\partial (W_{g,1})\dar\\
\BBlockDiff^\Gamma_\partial (W_{g,1})\rar&\BBlockDiff_\partial (W_{g,1})
\end{tikzcd}
\end{center}
for a fixed subgroup $\Gamma\le G_g$ to conclude the following.

\begin{prop}\label{appendix:prop1}For $2n\ge6$ and a subgroup $\Gamma\le G_g$, the induced map \[\oH^*(\BBlockDiff^\Gamma_\partial (W_{g,1}),\bfQ)\lra\oH^*(\BDiff^\Gamma_\partial (W_{g,1});\bfQ)\] is an isomorphism for $*< 2n-4$ and a monomorphism for $*=2n-4$.
\end{prop} 

\subsection*{\circled{$2$}}Our discussion of the second map in the composition \eqref{appendix:composition} is not specific to the manifold $W_{g,1}$, so we phrase it more general.

\begin{prop}\label{appendix:prop2}
For a compact, simply-connected, stably parallelisable manifold $M$ of dimension $d\ge 5$ with sphere boundary $\partial M\cong S^{d-1}$, the natural map
\[\BBlockDiff_\partial(M)\lra \BhAut^{\cong}_\partial(M)\] is split injective on rational cohomology rings.\end{prop}

\begin{proof}Following \cite[Sect.\,4.4]{berglund-madsen-rational}, we pick a base point $*\in\partial M$ in the boundary and denote by $\hAut^*_{\partial}(\tau_M^s)$ the topological monoid of homotopy automorphisms $f\colon M\ra M$ which are the identity on the boundary together with a bundle automorphism $\tilde{f}\colon \tau_M^s\rightarrow\tau_M^s$ of the stable tangent bundle which restricts to the identity over the fixed basepoint and whose underlying self-map of $M$ agrees with $f$. We denote by $\hAut^{*}_{\partial,\circ}(\tau_M^s)\subset \hAut^{*}_{\partial}(\tau_M^s)$ the kernel of the map $\hAut^*_{\partial}(\tau_M^s)\ra\pi_0\hAut_\partial (M)$ given by taking homotopy classes and forgetting the bundle map and by $\hAut^{*,\cong}_{\partial}(\tau_M^s)\subset \hAut^{*}_{\partial}(\tau_M^s)$ the preimage of the subgroup $\pi_0\hAut^{\cong}_\partial (M)\subset \pi_0\hAut_\partial (M)$ given by the image of $\BlockDiff_\partial (M)$ in $\hAut_\partial (M)$. Up weak equivalence, there is a canonical map of topological monoids $\BlockDiff_\partial (M)\ra \hAut^{*}_{\partial}(\tau_M^s)$ given by taking derivatives \cite[p.\,116]{berglund-madsen-rational}. This map induces a map of fibre sequences
\begin{center}\label{appendix:universalcover}
\begin{tikzcd}
\BBlockDiff_{\partial,\circ}(M)\arrow[r]\arrow[d]&\BBlockDiff_\partial(M)\arrow[r]\arrow[d]&\oB\pi_0\hAut^{\cong}_\partial(M)\arrow[d,equal]\\
\BhAut^{*}_{\partial,\circ}(\tau_M^s)\arrow[r]&\BhAut^{*,\cong}_\partial(\tau_M^s)\arrow[r]&\oB\pi_0\hAut^{\cong}_\partial(M),
\end{tikzcd}
\end{center}
where $\BlockDiff_{\partial, \circ}(M)\subset \BlockDiff_\partial(M)$ denotes kernel of the map \[\BlockDiff_\partial(M)\lra \pi_0\hAut^{\cong}_\partial(M).\] Berglund and Madsen showed that the two fibres in this diagram are nilpotent spaces \cite[Prop.\,4.8, Cor.\,4.14]{berglund-madsen-rational} and that the map between them is a rational equivalence \cite[Cor.\,4.21]{berglund-madsen-rational}. By an application of the Serre spectral sequence, this implies that the middle vertical map is a rational cohomology isomorphism, so to finish the proof, it suffices to show that the map $\BhAut^{*,\cong}_\partial(\tau_M^s)\ra \BhAut^{\cong}_\partial(M)$ admits a section. Fixing a trivialisation of $\tau_M^s$, every homotopy automorphism of $M$ is covered by a unique bundle automorphism that agrees with the identity on each fibre with respect to the chosen trivialisation, which induces a section of the map $\hAut^{*}_\partial(\tau_M^s)\ra \hAut_\partial(M)$. Restricting components and taking classifying spaces, this induces a section as wished.
\end{proof}

In the case $M=W_{g,1}$, the argument for \cref{appendix:prop2} shows more generally that $\BBlockDiff^{\Gamma}_\partial(M)\rightarrow \BhAut^{\cong,\Gamma}_\partial(M)$ is split injective on rational cohomology for any subgroup $\Gamma\le G_g$.

\subsection*{\circled{$3$}}
Combining the previous two propositions, we conclude that the map \[\oH^*(\BhAut^{\cong,\Gamma}_\partial(W_{g,1});\bfQ)\ra\oH^*(\BDiff^{\Gamma}_\partial(W_{g,1});\bfQ)\] is injective in degrees $*\le2n-4$, which finishes the proof of \cref{appendix:theorem} when combined with the following proposition.

\begin{prop}\label{appendix:prop4}For $2n\ge6$ and a finite index subgroup $\Gamma\le G_g$, the induced map
\[\oH^*(\oB \Gamma;\bfQ)\lra\oH^*(\BhAut^{\cong,\Gamma}_\partial (W_{g,1});\bfQ)\]
is an isomorphism for $*<n$ and a monomorphism for $*=n$. Moreover, if $g\ge 4-c$, then this map is an isomorphism for  $*<n+\min(n-1,g-1+c)$ and a monomorphism for $*=n+\min(n-1,g-1+c)$, where $c=0$ if $n$ is even and $c=1$ if $n$ is odd.
\end{prop}
\begin{proof}
We consider the rational Serre spectral sequence 
\begin{equation}\label{equ:sss}E^{p,q}_2=\oH^p(\oB\pi_0\hAut^{\cong,\Gamma}_\partial(W_{g,1});\oH^q(\BhAut^{\id}_\partial (W_{g,1});\bfQ))\implies \oH^{p+q}(\BhAut^{\cong,\Gamma}_\partial (W_{g,1});\bfQ).\end{equation}
of the fibration sequence 
\[\BhAut^{\id}_\partial (W_{g,1})\lra \BhAut^{\cong,\Gamma}_\partial (W_{g,1})\lra\oB\pi_0\hAut^{\cong,\Gamma}_\partial(W_{g,1}).\] By \cite[Prop.\,5.6]{berglund-madsen-rational}, the rational homotopy Lie algebra $\pi_{*+1}(\BhAut^{\id}_\partial (W_{g,1}))\otimes\bfQ$ is isomorphic, as a module over $\bfQ[\pi_0\hAut_\partial(W_{g,1})]$, to a graded sub Lie algebra of the Lie algebra $\Der^+\bfL(H)$ of positive degree derivations of the free graded Lie algebra $\bfL(H)$ on the graded vector space $H\coloneqq \oH_n(W_{g,1};\bfQ)\cong\bfQ^{2g}$ concentrated in degree $(n-1)$. In particular, the action of $\pi_0\hAut^{\cong,\Gamma}_\partial(W_{g,1})$ on $\pi_{*+1}(\BhAut^{\id}_\partial (W_{g,1}))\otimes\bfQ$ factors through the homology action \[\pi_0\hAut^{\cong,\Gamma}_\partial(W_{g,1})\lra \Gamma\subset  \GL_{2g}(\bfZ).\] Moreover, as a result of \cite[Prop.\,7.9]{berglund-madsen-rational}, there is an isomorphism of $\bfQ[\pi_0\hAut^{\cong,\Gamma}_\partial(W_{g,1})]$-modules
\[\oH^q(\BhAut^{\id}_\partial (W_{g,1});\bfQ)\cong \oH^{q}_{\CE}(\pi_{*+1}(\BhAut^{\id}_\partial (W_{g,1}))\otimes\bfQ),\] where the right hand side is the  Chevalley--Eilenberg cohomology of the graded homotopy Lie algebra.
Consequently, the action on these groups factors through $\Gamma$ as well. Since the homology action $\pi_0\hAut^{\cong,\Gamma}_\partial(W_{g,1})\ra \Gamma$ is surjective by definition and has finite kernel as a result of \cite[Prop.\,5.3]{berglund-madsen-rational}, it induces an identification of the $E^2$-page of \eqref{equ:sss} of the form
\begin{equation}\label{equ:E2}E^{p,q}_2=\oH^p(\oB\pi_0\hAut^{\cong,\Gamma}_\partial(W_{g,1});\oH^q(\BhAut^{\id}_\partial (W_{g,1});\bfQ))\cong \oH^p(\BGamma;\oH^q(\BhAut^{\id}_\partial (W_{g,1});\bfQ)).\end{equation}
As $H$ is concentrated in degree $(n-1)$, the Lie algebra $\Der^+\bfL(H)\cong \pi_{*+1}(\BhAut^{\id}_\partial (W_{g,1}))\otimes\bfQ$ is concentrated in degrees $k(n-1)$ for $k\ge1$. In particular the space $\BhAut^{\id}_\partial (W_{g,1})$ is rationally $(n-1)$-connected, so the $E_{2}$-page \eqref{equ:E2} vanishes for $q\le n-1$ except for the bottom row $p=0$, which implies the first claim. To prove the second part of the statement, note that in degrees $*< 2n-2$, the Lie algebra $\Der^+\bfL(H)$ is only nontrivial in degree $n-1$, so the reduced rational cohomology \[\widetilde{\oH}^q(\BhAut^{\id}_\partial (W_{g,1});\bfQ)\cong \widetilde{\oH}^{q}_{\CE}(\pi_{*+1}(\BhAut^{\id}_\partial (W_{g,1}))\otimes\bfQ)\] is for $*< 2n-1$ only nontrivial in degree $n$ where it is isomorphic to the dual $(\pi_{n}(\BhAut^{\id}_\partial (W_{g,1}))\otimes\bfQ)^\vee$. To prove the claim, it thus suffices to show that $\oH^p(\BGamma;(\pi_{n}(\BhAut^{\id}_\partial (W_{g,1}))\otimes\bfQ)^\vee)$ vanishes for $p< g-1+c$ and $g\ge4-c$, since this would imply that the $E^2$-page of \eqref{equ:sss} is for $p+q< n+\min(n-1,g-1+c)$ concentrated in the bottom row. By the discussion above, we have
\[(\pi_{n}(\BhAut^{\id}_\partial (W_{g,1}))\otimes\bfQ)^\vee\subset\Der^+\bfL(H)_{n-1}^\vee\cong \Hom(H,\bfL^2(H))^\vee\cong H\otimes \bfL^2(H)^\vee\subset H\otimes (H^{\otimes 2})^\vee\cong H^{\otimes 3} \] as a $\bfQ[\Gamma]$-module, where $\bfL^2(H)\subset \bfL(H)$ is the subspace spanned by brackets of length two. Here we used that $\bfL(H)$ is a free Lie algebra for the first isomorphism and the standard (symplectic or orthogonal) form for the last isomorphism. Since short exact sequences of finite dimensional $\bfQ[\Gamma]$-modules split (this uses $g\ge2$ and that $\bfQ$-algebraic groups $\Sp_{2g}(\bfQ)$ and $\oSO_{2g}(\bfQ)$ are simple, see for instance \cite[Thm\,A.1,Prop.\,B.4]{berglund-madsen-rational}), we conclude that $(\pi_{n}(\BhAut^{\id}_\partial (W_{g,1}))\otimes\bfQ)^\vee$ is a direct summand of $H^{\otimes 3}$, so it suffices to show that $\oH^p(\BGamma;H^{\otimes 3})$ vanishes for $p< g-1+c$ and $g\ge4-c$. By enhancements due to Tshishiku \cite[Thm 1.1, Thm 1.2]{tshishiku-borel-range} of Borel's vanishing ranges for the cohomology of arithmetic groups \cite{borel_cohoarith, borel_cohoarith2}, the cohomology of $\Gamma$ with coefficients in an irreducible finite-dimensional algebraic representation vanishes in this range, so we have $\oH^p(\BGamma;H^{\otimes 3})\cong \oH^p(\BGamma;\bfQ)\otimes (H^{\otimes 3})^{\Gamma}$, where $(-)^{\Gamma}$ denotes taking invariants. The proof will conclude by arguing that these invariants vanish. This can be simplified in two ways: firstly, if $n$ is even, after possibly restricting to the subgroup $\Gamma\cap \oSO_{g,g}(\bfZ)\le \Gamma$, we may assume that $\Gamma$ is contained in the connected algebraic subgroup $\oSO_{g,g}(\bfQ)\subset \oO_{g,g}(\bfQ)$ and secondly, it suffices to show that the invariants vanish after tensoring with $\bfR$. Since $\Sp_{2g}(\bfQ)$ and $\oSO_{g,g}(\bfQ)$ are connected, semi-simple and have no compact factors, the subgroup $\Gamma\subset \Sp_{2g}(\bfR)$ or $\Gamma\subset \oSO_{g,g}(\bfR)$ (depending on whether $n$ is even or odd) is a lattice \cite[Thm 7.8]{BorelHarish-Chandra} and moreover Zariski dense by Borel's density theorem \cite{BorelDensity}. This implies that the $\Gamma$-invariants of $H^{\otimes 3}\otimes\bfR$ agree with those of $\Sp_{2g}(\bfR)$ or $\oSO_{g,g}(\bfR)$, which in turn vanish by classical invariant theory.
\end{proof}

\begin{rem}\ 
\begin{enumerate}[(i)]
\item Note that the first part of the proof of \cref{appendix:prop4} did not require $\Gamma\le G_g$ to be of finite index. Since Propositions~\ref{appendix:prop1} and~\ref{appendix:prop2} are valid in this generality as well, \cref{appendix:theorem} holds for arbitrary subgroups $\Gamma\le G_g$ in the range $*\le n$ for $2n\ge6$ and $*\le 2$ for $2n=2$.
\item Similar arguments to those of \cref{appendix:prop4} were used in \cite{KrannichSlope1} to improve the ranges for rational homological stability of $\BhAut_\partial(W_{g,1})$, $\BBlockDiff_\partial(W_{g,1})$, and a truncation of $\BDiff_\partial(W_{g,1})$.
\end{enumerate}
\end{rem}

\bibliographystyle{alpha}
\bibliography{geocycles.bib,appendix.bib}

\begin{thebibliography}{GRW18}

\bibitem[And92]{anderson}
M.~T. Anderson.
\newblock The {$L^2$} structure of moduli spaces of {E}instein metrics on
  {$4$}-manifolds.
\newblock {\em Geom. Funct. Anal.}, 2(1):29--89, 1992.

\bibitem[ANP15]{avramidi-phan}
G.~Avramidi and T.~T\^am Nguyen-Phan.
\newblock Flat cycles in the homology of
  ${\Gamma}\backslash${SL}(m,$\mathbb{R})/{SO}(m)$.
\newblock {\em Comment. Math. Helv.}, 90(3):645--666, 2015.

\bibitem[Bes08]{besse}
A.~Besse.
\newblock {\em Einstein manifolds}.
\newblock Classics in Mathematics. Springer-Verlag, Berlin, 2008.
\newblock Reprint of the 1987 edition.

\bibitem[BHC62]{BorelHarish-Chandra}
A.~Borel and Harish-Chandra.
\newblock Arithmetic subgroups of algebraic groups.
\newblock {\em Ann. of Math. (2)}, 75:485--535, 1962.

\bibitem[BM13]{berglund-madsen-hs}
A.~Berglund and I.~Madsen.
\newblock Homological stability of diffeomorphism groups.
\newblock {\em Pure Appl. Math. Q.}, 9(1):1--48, 2013.

\bibitem[BM20]{berglund-madsen-rational}
A.~Berglund and I.~Madsen.
\newblock {Rational homotopy theory of automorphisms of manifolds}.
\newblock {\em Acta Math.}, 224(1):67--185, 2020.

\bibitem[Bor60]{BorelDensity}
A.~Borel.
\newblock Density properties for certain subgroups of semi-simple groups
  without compact components.
\newblock {\em Ann. of Math. (2)}, 72:179--188, 1960.

\bibitem[Bor74]{borel_cohoarith}
A.~Borel.
\newblock Stable real cohomology of arithmetic groups.
\newblock {\em Ann. Sci. \'Ecole Norm. Sup. (4)}, 7:235--272 (1975), 1974.

\bibitem[Bor81]{borel_cohoarith2}
A.~Borel.
\newblock Stable real cohomology of arithmetic groups. {II}.
\newblock In {\em Manifolds and {L}ie groups ({N}otre {D}ame, {I}nd., 1980)},
  volume~14 of {\em Progr. Math.}, pages 21--55. Birkh\"auser, Boston, Mass.,
  1981.

\bibitem[BP17]{botvinnik-perlmutter}
B.~Botvinnik and N.~Perlmutter.
\newblock Stable moduli spaces of high-dimensional handlebodies.
\newblock {\em J. Topol.}, 10(1):101--163, 2017.

\bibitem[DK01]{davis-kirk}
J.~F. Davis and P.~Kirk.
\newblock {\em Lecture notes in algebraic topology}, volume~35 of {\em Graduate
  Studies in Mathematics}.
\newblock American Mathematical Society, Providence, RI, 2001.

\bibitem[Don90]{donaldson}
S.~K. Donaldson.
\newblock Polynomial invariants for smooth four-manifolds.
\newblock {\em Topology}, 29(3):257--315, 1990.

\bibitem[ERW15]{EbertRW}
J.~Ebert and O.~Randal-Williams.
\newblock Torelli spaces of high-dimensional manifolds.
\newblock {\em J. Topol.}, 8(1):38--64, 2015.

\bibitem[FM88]{friedman-morgan}
R.~Friedman and J.~W. Morgan.
\newblock On the diffeomorphism types of certain algebraic surfaces. {I}.
\newblock {\em J. Differential Geom.}, 27(2):297--369, 1988.

\bibitem[FOR00]{for}
F.~T. Farrell, P.~Ontaneda, and M.~S. Raghunathan.
\newblock Non-univalent harmonic maps homotopic to diffeomorphisms.
\newblock {\em J. Differential Geom.}, 54(2):227--253, 2000.

\bibitem[Gia09]{giansiracusa}
J.~Giansiracusa.
\newblock The diffeomorphism group of a {$K3$} surface and {N}ielsen
  realization.
\newblock {\em J. Lond. Math. Soc. (2)}, 79(3):701--718, 2009.

\bibitem[GRW14]{GRW-stable-moduli}
S.~Galatius and O.~Randal-Williams.
\newblock Stable moduli spaces of high-dimensional manifolds.
\newblock {\em Acta Math.}, 212(2):257--377, 2014.

\bibitem[GRW17]{GRW-homological-stability-2}
S.~Galatius and O.~Randal-Williams.
\newblock Homological stability for moduli spaces of high dimensional
  manifolds. {II}.
\newblock {\em Ann. of Math. (2)}, 186(1):127--204, 2017.

\bibitem[GRW18]{GRW-homological-stability-1}
S.~Galatius and O.~Randal-Williams.
\newblock Homological stability for moduli spaces of high dimensional
  manifolds. {I}.
\newblock {\em J. Amer. Math. Soc.}, 31(1):215--264, 2018.

\bibitem[Hai02]{Hain3folds}
R.~Hain.
\newblock The rational cohomology ring of the moduli space of abelian 3-folds.
\newblock {\em Math. Res. Lett.}, 9(4):473--491, 2002.

\bibitem[Har85]{Harer}
J.~L. Harer.
\newblock Stability of the homology of the mapping class groups of orientable
  surfaces.
\newblock {\em Ann. of Math. (2)}, 121(2):215--249, 1985.

\bibitem[HT12]{HulekTomassiVoronoi}
K.~Hulek and O.~Tommasi.
\newblock Cohomology of the second {V}oronoi compactification of
  {$\mathcal{A}_4$}.
\newblock {\em Doc. Math.}, 17:195--244, 2012.

\bibitem[Kra20]{KrannichSlope1}
M.~Krannich.
\newblock A note on rational homological stability of automorphisms of
  manifolds.
\newblock {\em Q. J. Math.}, 71(3):1069--1079, 2020.

\bibitem[KRW20]{KupersRandalWilliams}
A.~Kupers and O.~Randal-Williams.
\newblock {The cohomology of Torelli groups is algebraic}.
\newblock {\em Forum of Mathematics, Sigma}, 2020.
\newblock to appear.

\bibitem[Loo81]{looijenga}
E.~Looijenga.
\newblock A {T}orelli theorem for {K}\"ahler-{E}instein {$K3$} surfaces.
\newblock volume 894 of {\em Lecture Notes in Math.}, pages 107--112. Springer,
  Berlin-New York, 1981.

\bibitem[LS86]{lee-schwermer}
R.~Lee and J.~Schwermer.
\newblock Geometry and arithmetic cycles attached to {${\rm SL}_3({\bf Z})$}.
  {I}.
\newblock {\em Topology}, 25(2):159--174, 1986.

\bibitem[Mat86]{matumoto_k3}
T.~Matumoto.
\newblock On diffeomorphisms of a {$K3$} surface.
\newblock In {\em Algebraic and topological theories ({K}inosaki, 1984)}, pages
  616--621. Kinokuniya, Tokyo, 1986.

\bibitem[MH73]{milnor-husemoller}
J.~Milnor and D.~Husemoller.
\newblock {\em Symmetric bilinear forms}.
\newblock Springer-Verlag, New York-Heidelberg, 1973.
\newblock Ergebnisse der Mathematik und ihrer Grenzgebiete, Band 73.

\bibitem[Mil76]{millson}
J.~Millson.
\newblock On the first {B}etti number of a constant negatively curved manifold.
\newblock {\em Ann. of Math. (2)}, 104(2):235--247, 1976.

\bibitem[Mor87]{morita-ccs}
S.~Morita.
\newblock Characteristic classes of surface bundles.
\newblock {\em Invent. Math.}, 90(3):551--577, 1987.

\bibitem[Mor15]{morris}
D.~Witte Morris.
\newblock {\em Introduction to arithmetic groups}.
\newblock Deductive Press, Place of publication not identified, 2015.

\bibitem[MR80]{millson-raghunathan}
J.~Millson and M.~S. Raghunathan.
\newblock Geometric construction of cohomology for arithmetic groups. {I}.
\newblock In {\em Geometry and analysis}, pages 103--123. Indian Acad. Sci.,
  Bangalore, 1980.

\bibitem[MW07]{MadsenWeiss}
I.~Madsen and M.~Weiss.
\newblock The stable moduli space of {R}iemann surfaces: {M}umford's
  conjecture.
\newblock {\em Ann. of Math. (2)}, 165(3):843--941, 2007.

\bibitem[O'M00]{omeara}
O.~T. O'Meara.
\newblock {\em Introduction to quadratic forms}.
\newblock Classics in Mathematics. Springer-Verlag, Berlin, 2000.
\newblock Reprint of the 1973 edition.

\bibitem[PR72]{prasad-raghunathan}
G.~Prasad and M.~S. Raghunathan.
\newblock Cartan subgroups and lattices in semi-simple groups.
\newblock {\em Ann. of Math. (2)}, 96:296--317, 1972.

\bibitem[RS93]{rohlfs-schwermer}
J.~Rohlfs and J.~Schwermer.
\newblock Intersection numbers of special cycles.
\newblock {\em J. Amer. Math. Soc.}, 6(3):755--778, 1993.

\bibitem[RW17]{UpperBound}
O.~Randal-Williams.
\newblock An upper bound for the pseudoisotopy stable range.
\newblock {\em Math. Ann.}, 368(3-4):1081--1094, 2017.

\bibitem[Sch10]{schwermer-survey}
J.~Schwermer.
\newblock Geometric cycles, arithmetic groups and their cohomology.
\newblock {\em Bull. Amer. Math. Soc. (N.S.)}, 47(2):187--279, 2010.

\bibitem[Sco05]{scorpan}
A.~Scorpan.
\newblock {\em The wild world of 4-manifolds}.
\newblock American Mathematical Society, Providence, RI, 2005.

\bibitem[Sul77]{sullivan-infinitesimal}
D.~Sullivan.
\newblock Infinitesimal computations in topology.
\newblock {\em Inst. Hautes \'Etudes Sci. Publ. Math.}, (47):269--331 (1978),
  1977.

\bibitem[Thu74]{thurston-diff}
W.~Thurston.
\newblock Foliations and groups of diffeomorphisms.
\newblock {\em Bull. Amer. Math. Soc.}, 80:304--307, 1974.

\bibitem[Tsh19]{tshishiku-borel-range}
B.~Tshishiku.
\newblock Borel's stable range for the cohomology of arithmetic groups.
\newblock {\em J. Lie Theory}, 29(4):1093--1102, 2019.

\bibitem[Ven08]{venkataramana-betti}
T.~N. Venkataramana.
\newblock Virtual {B}etti numbers of compact locally symmetric spaces.
\newblock {\em Israel J. Math.}, 166:235--238, 2008.

\bibitem[Wal64]{wall_diff4mfld}
C.~T.~C. Wall.
\newblock Diffeomorphisms of {$4$}-manifolds.
\newblock {\em J. London Math. Soc.}, 39:131--140, 1964.

\end{thebibliography}

\end{document}